\documentclass[12pt]{amsart}
\usepackage{amsmath}
\usepackage{amssymb}
\usepackage{amsfonts}
\usepackage{textcomp}
\usepackage{amsthm}
\usepackage{mathrsfs}
\usepackage[latin1]{inputenc}
\usepackage[all]{xy}
\usepackage{hyperref}
\usepackage{url}
\usepackage{graphicx}
\usepackage{enumitem}
\usepackage{tikz-cd}
\usepackage[paperheight=11in,paperwidth=8.5in,left=1in,top=1.25in,right=1in,bottom=1.25in,]{geometry}

\newtheorem{theorem}{Theorem}[section]
\newtheorem{lem}{Lemma}[section]

\newtheorem{defi}{Definition}[section]

\newtheorem*{maintheorem}{Main Theorem}

\newcommand{\PSLC}{{\rm PSL}_2\mathbb{C}}
\newcommand{\SLC}{{\rm SL}_2\mathbb{C}}

\title{Peripheral birationality for 3-dimensional convex co-compact $\PSLC$ varieties}
\author{Ian Agol, Franco Vargas Pallete}
\email{}
\thanks{I. Agol's research supported by a Simons Investigator Grant and the Mathematical Sciences Research Institute. F. Vargas Pallete's research was supported by NSF grant DMS-2001997.}

\address{}
\begin{document}

\maketitle

\begin{abstract}
    Let $M$ be a hyperbolizable $3$-manifold with boundary, and let $\chi_0(M)$ be a component of the $\PSLC$-character variety of $M$ that contains the convex co-compact characters. We show that the peripheral map $i_*:\chi_0(M)\rightarrow\chi(\partial M)$ to the character variety of $\partial M$ is a birational isomorphism with its image, and in particular is generically a one-to-one map. This generalizes work of Dunfield (one cusped hyperbolic $3$-manifolds) and Klaff-Tillmann (finite volume hyperbolic $3$-manifolds). We use the Bonahon-Schl\"afli formula and volume rigidity of discrete co-compact representations.
\end{abstract}

\section{Introduction}\label{sec:introduction}
Given a connnected 3-manifold $M$ with boundary and a representation $\rho:\pi_1(M)\to G$, $G$ a Lie group (which will usually be $PSL_2(\mathbb{C})$), it is natural to ask to what extent $\rho$ is determined by $\rho_{|\pi_1(\partial M)}$? If the interior of $M$ admits a convex cocompact hyperbolic metric with holonomy $\rho$, then it is known that $\rho_{|\pi_1(\partial M)}$ determines $\rho$ - in fact, the conformal structure of the boundary determines the hyperbolic metric by results of Ahlfors-Bers. For a manifold $M$ with one cusp, Dunfield \cite{Dunfield99} proved that a main component of the character variety (the Zariski component containing a discrete faithful representation) maps birationally to a factor of the $A$-polynomial. In particular, for a Zariski open subset of this component, the represenation will be determined by its restriction to the boundary. This was extended by Klaff and Tillmann \cite{KlaffTillmann} to the multiple cusped case (see also Francaviglia \cite{Francaviglia}, Francaviglia-Klaff \cite{FrancavigliaKlaff}). 
However, in general there may be families of representations which are constant on $\pi_1\partial M$. 

We generalize the result of Dunfield to convex cocompact hyperbolic manifolds. In this case, the space of discrete faithful representations has dimension determined by Teichmuller space, and this is half the dimension of the character variety of the boundary. Because of the dimension agreeing, we know that the map is generically finite-to one. We show that the map from the main component of the character variety to its image in the character variety of the boundary is a birational map, so is actually generically one-to-one. 

The tools that Dunfield uses are the Schl\"afli formula and representation rigidity of Gromov which we adapt to the infinite volume case. An obvious issue with geometrically finite manifolds of negative euler characteristic is that the volume is infinite. There is the notion of convex core volume, but this is only well-defined for discrete faithful representations. Rather than try to extend this to non-discrete representations, we choose pleated surfaces with the same bending locus to define a notion of volume which depends only on the restriction of the representation to the boundary. Although such volumes require some choices, and are not defined everywhere, nevertheless we can show that there is a notion of volume determined by the restriction to the boundary, making use of a version of Schl\"afli's formula due to Bonahon \cite{Bonahon98}. The other tool, representation rigidity, then is proved at countably infinitely many representations for which there is an extension to a finite-volume hyperbolic orbifold representation. One could probably also extend the proof of volume rigidity to all representations whose restriction to the boundary is discrete and faithful, but rather than prove such a result, we decided to use what tools were already at hand. Once we have extended these two tools, the proof proceeds similarly to that of Dunfield. Suppose that a geometrically finite representation and another representation have the same peripheral holonomy. Then the volumes of the representations are the same. Hence by volume rigidity, they are both discrete, and hence conjugate.

\subsection{Examples} 
Here are some examples for which the result can be proved more easily. Suppose one has a compression body. Then the fundamental group is a free product of surface groups. There is a compressible boundary component which surjects the fundamental group, and hence the representation variety embeds in the representation variety of this boundary component. 

A slightly less non-trivial example is that of a book-of-I-bundles. In this case, when the representation is faithful, it determines the representations of each page. For each binding, there will be boundary components overlapping the pages meeting that binding. The representations of these boundary components determine how to ``glue" together the representations of adjacent pages. Hence the boundary holonomy determines the full representation generically (since faithful representations are generic in the main component). 

\subsection{Outline}

This paper is organized as follows. Section \ref{sec:background} explains the main tools we will use. Subsection \ref{subsec:character} reviews $\PSLC$ character varieties. Subsection \ref{subsec:pleated} constructs the pleated surfaces we will use to compute the volume associate to a character. Subsections \ref{subsec:bendingcocycle}, \ref{subsec:volumevariation} deal with the definition of bending cocycle and its implementation in the Bonahon-Schl\"afli formula for change of volume. Section \ref{sec:mainproof} uses all these tools and volume rigidity of co-compact characters to prove our main result (stated below). We finish with some remarks where we discuss the surjectivity and finite-to-one nature of the peripheral map, and also how to generalize the result for geometrically finite characters.

\begin{maintheorem}
Let $M$ be a hyperbolizable compact 3-manifold with boundary. Let $\chi_0(M)$ be the connected component of the discrete and faithful representations. Then the map $i_*:\chi_0(M)\rightarrow \chi(\partial M)$ is a birational isomorphism onto its image.
\end{maintheorem}

\section*{Acknowledgments} Both authors are thankful to F. Bonahon for its interest and helpful comments.

\section{Background}\label{sec:background}

\subsection{Character variety}\label{subsec:character}

For a comprehensive study of character varieties of $3$-manifold groups, we refer the reader to \cite{CullerShalen83} for $\SLC$ characters and \cite{BoyerZhang98} for $\PSLC$ characters. Here we present a factual recollection of the relevant definitions and results from \cite{BoyerZhang98}.

Let $G$ be a finitely generated group. A $\PSLC$-representation is a homomorphism $\rho:G\rightarrow \PSLC$. The $\PSLC$ \textit{representation variety} $R(G)$ is defined by 
\[R(G) = \lbrace \rho\,|\,\rho:G\rightarrow\PSLC \text{  homomorphism} \rbrace
\]

In order to discuss the algebraic structure on $R(G)$, we recall some definitions and properties from algebraic geometry (see \cite{Harris92}). An \textit{affine algebraic set} in $\mathbb{C}^m$ is the zero locus of a finite collection of polynomials with complex coefficients. Given $U, V$ affine algebraic sets, we say that $f:U\rightarrow V$ is a \textit{regular map} if $f:U\rightarrow \mathbb{C}^m\supseteq V$ has polynomial coordinates and $f(U)\subseteq V$. Regular maps are in bijection (by taking the pull-back) with homomorphism between the coordinate rings of regular functions $A[V]\rightarrow A[U]$. 

We say that two affine algebraic sets $U\subseteq \mathbb{C}^m ,V\subseteq \mathbb{C}^n$ are \text{isomorphic} if there exists regular maps $f:U\rightarrow V,\, g:V\rightarrow U$ with polynomial coordinates so that $g\circ f = id_U, f\circ g= id_V$. An affine algebraic set is called \textit{irreducible} if it is the zero locus of a finite collection of polynomials that generate a prime ideal. Every affine algebraic set $U$ is canonically decomposed (respecting isomorphisms) as the finite union of affine algebraic varieties, each of which is known as an \textit{irreducible component} of $U$.

The adjoint representation $\PSLC \rightarrow {\rm Aut} (sl_2\mathbb{C})$ realizes $\PSLC$ as an affine algebraic set. Hence if we take $\lbrace g_1, \ldots, g_n\rbrace$ a collection of generators of $G$, the map $R(G) \rightarrow (\PSLC)^n$, $\rho\mapsto (\rho(g_1),\ldots,\rho(g_n))$ identifies $R(G)$ with an affine algebraic set. Since a different choice of generators will give an isomorphic affine algebraic set, we identify $R(G)$ with an isomorphism class of affine algebraic sets.

We say two representations $\rho_1, \rho_2\in R(G)$ are \textit{conjugated} if there exists $g\in {\rm GL}_2\mathbb{C}$ so that $\rho_2 = g^{-1}\rho_1 g$. This defines an equivalence relation in $R(G)$, where any pair of equivalent representations belong to the same irreducible component of $R(G)$. 

We say that a representation $\rho\in R(G)$ is \textit{irreducible} if the only subspaces of $\mathbb{C}^2$ invariant by $\rho(G)$ are $\lbrace0\rbrace$ and $\mathbb{C}^2$, otherwise we say that $\rho$ is \textit{reducible}. Irreducibility is preserved by conjugation. Moreover, the set of reducible representations is a subvariety of $R(G)$. Hence by an irreducible component of $R(G)$ of irreducible representations we refer to an irreducible component of $R(G)$ (in the algebro-geometric sense) so that the subvariety of reducible representations is proper.

We define the \textit{character variety} $\chi(G)$ as the algebro-geometric quotient $ R(G)/\PSLC$ by considering the affine algebraic set matching the subring $A^{\PSLC}[R(G)]\subseteq A[R(G)]$ of regular functions invariant by the natural $\PSLC$ action. Hence we have a surjective regular map $\chi:R(G)\rightarrow\chi(G)$ that is constant in the $\PSLC$ orbits, and for any $g\in G$ we have that the map $\tau_g: \chi(G) \rightarrow \mathbb{C}$, $\chi_\rho\mapsto (tr(\rho(g)))^2$ is a well-defined regular map.

We introduce some definitions for rational maps between affine algebraic varieties.

\begin{defi}
Let $X\subseteq \mathbb{C}^m, Y\subseteq \mathbb{C}^n$ be affine algebraic varieties. We say that $\phi=(\phi_1,\ldots,\phi_n)$ is a rational map from $X$ to $Y$ if each $\phi_1,\ldots,\phi_n$ is given by a rational function in $X$, and whenever defined, $\phi(x)\in Y$. We denote this by $\phi:X\dashrightarrow Y$. We say that $\phi:X\dashrightarrow Y$ is dominant if, given $U\subseteq X$ Zariski open set where $\phi$ is defined, then the Zariski closure of $\phi(U)$ is equal to $Y$. We say that $\phi$ is birational if there exists inverse $\psi:Y\dashrightarrow X$. In such case we say that $X$ and $Y$ are birationally isomorphic.
\end{defi}

Now we consider the convex co-compact $\PSLC$ representations of a $3$-manifold $M$ with boundary. This representations are irreducible and it is known (see for instance \cite[Section 6]{DumasKent}) that their Zariski closure in $R(\pi_1(M))$ is an irreducible component $R_0$, so their characters have Zariski closure an irreducible component of $\chi(M)$. We denote this component by $\chi_0(M)$. By fixing paths from a basepoint to the each component of the boundary $\partial M = \Sigma_1\sqcup\ldots\sqcup\Sigma_k$ we have the regular maps between representation varieties induced by the inclusion $\partial M \hookrightarrow M$

\[{i_\ell}_*:R(\pi_1(M)) \rightarrow R(\pi_1(\Sigma_\ell)), 1\leq \ell\leq k.
\]
Hence by taking the pullback we have ring homomorphisms 
\[\varphi_\ell:A[R(\pi_1(\partial M))] \rightarrow A[R(\pi_1(\Sigma_\ell))].
\]
It is not hard to see that $\varphi_\ell$ is equivariant with respect to the natural $\PSLC$ actions, so in particular satisfies $\varphi_\ell(A^{\PSLC}[R(\pi_1(\Sigma_\ell))])\subseteq A^{\PSLC}[R(\pi_1(M))]$. This means that we have regular maps (which we also denote by ${i_\ell}_*$) ${i_\ell}_*:\chi(\pi_1(M))\rightarrow \chi(\pi_1(\Sigma_\ell))$ that makes the following diagram commute.

\[ \begin{tikzcd}
R(\pi_1(M)) \arrow{r}{{i_\ell}_*} \arrow[swap]{d}{\chi} & R(\pi_1(\Sigma_\ell)) \arrow{d}{\chi} \\%
\chi(\pi_1(M)) \arrow{r}{{i_\ell}_*}& \chi(\pi_1(\Sigma_\ell))
\end{tikzcd}
\]

We should observe that while each ${i_\ell}_*$ at the level of representation varieties depended on the choice of basepoint and paths in $M$, the maps ${i_\ell}_*$ at the level of character varieties are well-defined, since a change of paths conjugates representations. We define then the \textit{peripheral map} of $\chi(M)$ as the regular map

\[i_*=({i_1}_*,\ldots,{i_k}_*):\chi(\pi_1(M)) \rightarrow \chi(\partial M):= \chi(\pi_1(\Sigma_1))\times\ldots\times\chi(\pi_1(\Sigma_k))
\]

%Likewise we can define a natural map at the level of characters, which we also denote by $i_*$
%\[i_*:\chi(\pi_1(M)) \rightarrow \chi(\pi_1(\partial M)) = \chi(\pi_1(\Sigma_1))\times \ldots \times \chi(\pi_1(\Sigma_n))
%\]

%Observe that by including the generators of $T(\pi_1(\Sigma_1)),\ldots T(\pi_1(\Sigma_k))$ in the collection of generators of $T(\pi_1(M))$ we have affine charts where $i_*:\chi(\pi_1(M)) \rightarrow \chi(\pi_1(\partial M))$ is represented by coordinate projections. Hence $i_*$ is a rational map.

By Ahlfors-Bers \cite{AhlforsBers60}  we have that $i_*$ evaluated at a convex co-compact $\PSLC$ characters of $\pi_1(M)$ gives a collection of convex co-compact $\PSLC$ characters for each $\Sigma_\ell$. These collection of characters have closure included in the convex co-compact irreducible component of $\chi(\partial M)$, which we denote by $\chi_0(\partial M)$. Observe that the map $i_*: \chi_0(M)\rightarrow  \overline{i_*(\chi_0(M))}\subset \chi_0(\partial M)$ is a diffeomorphism between convex co-compact characters and their image. It is well-known (see for instance \cite[Chapter I.3]{Hartshorne}) that ${\rm dim}_{\mathbb{C}}(i_*(\chi_0(M))) \leq {\rm dim}_{\mathbb{C}}(\chi_0(M))$, as dimension corresponds to the Krull dimension of local rings, and since $i_*$ has dense image then the induced map at the level of local rings is injective. Since dimension agrees with the dimension of smooth points (such as the convex co-compact characters), we know that the map

\[i_*:\chi_0(M) \rightarrow \overline{i_*(\chi_0(M))}\subset \chi_0(\partial M)
\]
is a dominant (i.e. dense image) regular map between algebraic varieties of the same dimension.

\subsection{Pleated surfaces}\label{subsec:pleated}

Let $S$ be a genus $g$ surface. Let us fix an auxiliary hyperbolic metric $m_0$ and $\Gamma=\lbrace\gamma_i \rbrace_{1\leq i\leq 3g-3}$ a maximal collection of oriented, disjoint, essential, pairwise non-isotopic, simple closed curves (i.e. an oriented pants decomposition). Take $\lambda_0$ to be the maximal lamination extension of $\Gamma$ so that any leaf of $\lambda_0$ that accumulates at $\gamma_i\in\Gamma$ does it in the direction of the orientation (see Figure \ref{fig:lamination}). Equivalently, the orientation for each $\gamma_i\in\Gamma$ gives a preferred endpoint for any lift $\tilde{\gamma_i}$ of $\gamma_i$ in the universal cover. Then any ideal triangle in the lift $(\tilde{S},\tilde{\lambda_0})$ uses only preferred endpoints as ideal vertices.

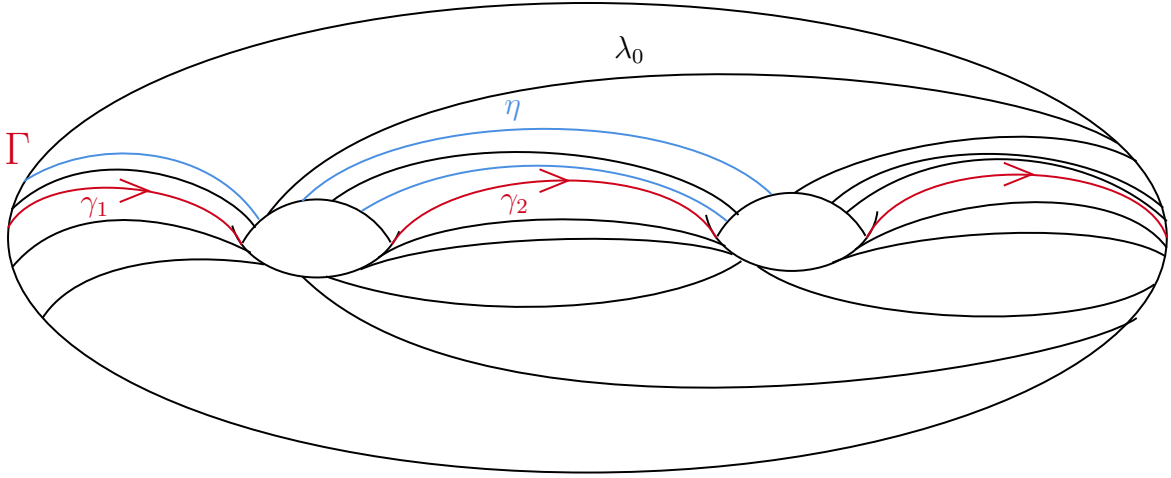
\begin{figure}[hbt!]
    \centering

\tikzset{every picture/.style={line width=0.75pt}} %set default line width to 0.75pt        

\begin{tikzpicture}[x=0.75pt,y=0.75pt,yscale=-1,xscale=1]

%Shape: Ellipse [id:dp06498004002151192] 
\draw   (7.59,123.5) .. controls (7.59,58.05) and (138.42,5) .. (299.8,5) .. controls (461.18,5) and (592,58.05) .. (592,123.5) .. controls (592,188.95) and (461.18,242) .. (299.8,242) .. controls (138.42,242) and (7.59,188.95) .. (7.59,123.5) -- cycle ;
%Shape: Arc [id:dp3134874249123625] 
\draw  [draw opacity=0] (204.9,120.21) .. controls (199.53,133.89) and (182.63,143.76) .. (162.75,143.53) .. controls (141.99,143.28) and (124.78,132.1) .. (120.66,117.37) -- (163.21,110.73) -- cycle ; \draw   (204.9,120.21) .. controls (199.53,133.89) and (182.63,143.76) .. (162.75,143.53) .. controls (141.99,143.28) and (124.78,132.1) .. (120.66,117.37) ;
%Shape: Arc [id:dp2183265689474243] 
\draw  [draw opacity=0] (446.04,110.16) .. controls (444.29,127.22) and (425.26,140.5) .. (402.24,140.24) .. controls (380.49,139.98) and (362.63,127.72) .. (359.65,111.95) -- (402.69,107.44) -- cycle ; \draw   (446.04,110.16) .. controls (444.29,127.22) and (425.26,140.5) .. (402.24,140.24) .. controls (380.49,139.98) and (362.63,127.72) .. (359.65,111.95) ;
%Shape: Arc [id:dp17312565680083725] 
\draw  [draw opacity=0] (125.84,126.06) .. controls (133.24,112.75) and (147.74,103.84) .. (164.2,104.17) .. controls (179.95,104.48) and (193.51,113.17) .. (200.57,125.82) -- (163.21,146.03) -- cycle ; \draw   (125.84,126.06) .. controls (133.24,112.75) and (147.74,103.84) .. (164.2,104.17) .. controls (179.95,104.48) and (193.51,113.17) .. (200.57,125.82) ;
%Shape: Arc [id:dp6208828908356181] 
\draw  [draw opacity=0] (365.33,122.77) .. controls (372.72,109.46) and (387.22,100.55) .. (403.68,100.87) .. controls (419.43,101.19) and (432.99,109.88) .. (440.05,122.53) -- (402.7,142.74) -- cycle ; \draw   (365.33,122.77) .. controls (372.72,109.46) and (387.22,100.55) .. (403.68,100.87) .. controls (419.43,101.19) and (432.99,109.88) .. (440.05,122.53) ;
\draw  [color={rgb, 255:red, 208; green, 2; blue, 27 }  ,draw opacity=1 ] (63.92,93.98) -- (78.21,99.74) -- (63.92,105.5) ;
\draw  [color={rgb, 255:red, 208; green, 2; blue, 27 }  ,draw opacity=1 ] (275.67,88.94) -- (289.97,94.7) -- (275.67,100.46) ;
\draw  [color={rgb, 255:red, 208; green, 2; blue, 27 }  ,draw opacity=1 ] (509.79,85.65) -- (524.09,91.41) -- (509.79,97.17) ;
%Curve Lines [id:da8144573066132244] 
\draw [color={rgb, 255:red, 74; green, 144; blue, 226 }  ,draw opacity=1 ]   (156,105) .. controls (199,59) and (345,55) .. (392.71,101.65) ;
%Curve Lines [id:da33928232725045615] 
\draw [color={rgb, 255:red, 74; green, 144; blue, 226 }  ,draw opacity=1 ]   (186,110) .. controls (241,75) and (332,83) .. (370,115) ;
%Curve Lines [id:da8048831572366224] 
\draw [color={rgb, 255:red, 74; green, 144; blue, 226 }  ,draw opacity=1 ]   (15.67,94.93) .. controls (49,73) and (108,75) .. (134.17,114.26) ;
%Curve Lines [id:da552693119377796] 
\draw [color={rgb, 255:red, 208; green, 2; blue, 27 }  ,draw opacity=1 ]   (7.59,118.46) .. controls (24.65,85.68) and (112.62,95.77) .. (125.32,127.03) ;
%Curve Lines [id:da8927786060589591] 
\draw [color={rgb, 255:red, 208; green, 2; blue, 27 }  ,draw opacity=1 ]   (201.13,126.86) .. controls (225.73,86.52) and (343.33,82.32) .. (364.8,123.74) ;
%Curve Lines [id:da19583390262924616] 
\draw [color={rgb, 255:red, 208; green, 2; blue, 27 }  ,draw opacity=1 ]   (440.61,123.57) .. controls (464.53,77.28) and (580.33,84.84) .. (592,123.5) ;
%Curve Lines [id:da8160541657948299] 
\draw [color={rgb, 255:red, 0; green, 0; blue, 0 }  ,draw opacity=1 ]   (138.66,111.73) .. controls (213.17,10.88) and (502.23,36.1) .. (566.86,75.6) ;
%Curve Lines [id:da4766503800572357] 
\draw [color={rgb, 255:red, 0; green, 0; blue, 0 }  ,draw opacity=1 ]   (9.39,109.21) .. controls (47.99,71.39) and (116.21,94.09) .. (132.37,118.46) ;
%Curve Lines [id:da8890852643190366] 
\draw [color={rgb, 255:red, 0; green, 0; blue, 0 }  ,draw opacity=1 ]   (423.23,105.85) .. controls (462.73,69.71) and (553.4,75.6) .. (590.2,108.37) ;
%Curve Lines [id:da38771831918185184] 
\draw [color={rgb, 255:red, 0; green, 0; blue, 0 }  ,draw opacity=1 ]   (404.38,100.81) .. controls (440.29,75.6) and (535.44,60.47) .. (576.74,84.84) ;
%Curve Lines [id:da1390217194014971] 
\draw [color={rgb, 255:red, 0; green, 0; blue, 0 }  ,draw opacity=1 ]   (171,105) .. controls (225,64) and (345.48,79.22) .. (376,112) ;
%Curve Lines [id:da12605311619122395] 
\draw [color={rgb, 255:red, 0; green, 0; blue, 0 }  ,draw opacity=1 ]   (430.41,110.89) .. controls (477.09,70.55) and (553.4,78.12) .. (592,115.1) ;
%Curve Lines [id:da9223165663446431] 
\draw    (168,143) .. controls (239,168) and (341.54,160.48) .. (377.45,135.27) ;
%Curve Lines [id:da014147163543684727] 
\draw    (156,143) .. controls (252,238) and (540.93,189.21) .. (576.84,164) ;
%Curve Lines [id:da21614511056280872] 
\draw    (385.53,137.79) .. controls (425.03,167.2) and (549.81,172.24) .. (585.72,147.03) ;
%Curve Lines [id:da6768108995995552] 
\draw    (25,164) .. controls (43,136) and (94.67,130.22) .. (136.86,136.95) ;
%Curve Lines [id:da46993355400578385] 
\draw    (10,138) .. controls (41,101) and (99.48,115.03) .. (130,131) ;
%Curve Lines [id:da696867249780365] 
\draw    (194.31,133.59) .. controls (230.22,108.37) and (330.77,109.21) .. (369.37,127.7) ;
%Curve Lines [id:da05619687172471832] 
\draw    (185.34,139.47) .. controls (226.63,121.82) and (340.64,119.3) .. (372.96,131.9) ;
%Curve Lines [id:da7446809162254087] 
\draw    (434.9,129.38) .. controls (477.09,101.65) and (566.86,94.09) .. (592,128.54) ;
%Curve Lines [id:da41676712547007955] 
\draw    (423.23,136.11) .. controls (463.63,119.3) and (561.48,114.26) .. (591.1,132.74) ;

% Text Node
\draw (4.08,68.96) node [anchor=north west][inner sep=0.75pt]  [font=\Large]  {$\textcolor[rgb]{0.82,0.01,0.11}{\Gamma }$};
% Text Node
\draw (311.79,20.37) node [anchor=north west][inner sep=0.75pt]   [align=left] {$\displaystyle \lambda _{0}$};
% Text Node
\draw (42.63,101.73) node [anchor=north west][inner sep=0.75pt]  [color={rgb, 255:red, 208; green, 2; blue, 27 }  ,opacity=1 ]  {$\gamma _{1}$};
% Text Node
\draw (254.39,100.21) node [anchor=north west][inner sep=0.75pt]  [color={rgb, 255:red, 208; green, 2; blue, 27 }  ,opacity=1 ]  {$\gamma _{2}$};
% Text Node
\draw (256.54,52.87) node [anchor=north west][inner sep=0.75pt]  [color={rgb, 255:red, 74; green, 144; blue, 226 }  ,opacity=1 ]  {$\eta $};

\end{tikzpicture}
    \caption{Lamination $\lambda_0$ defined by the collection of oriented curves $\Gamma$. The component $\eta$ of $\lambda_0$ accumulates at $\gamma_1, \gamma_2$. according to their orientation.}
    \label{fig:lamination}
\end{figure}\noindent

Take representations $\rho_t:\pi_1(S)\rightarrow PSL(2,\mathbb{C})$ so that $\rho(\gamma_i)$ is neither parabolic or the identity (so there is a well-defined axis), and the endpoints of $\rho(\gamma_i), \rho(\gamma_j)$ are distinct for $i\neq j$. Assume as well that we have a given equivariant orientation/endpoint for axis of each lift of $\gamma_i$, which we denote by $\zeta$. Then we define the (generalized) pleated surface $(\tilde{f}_{\zeta} = \tilde{f}:\tilde{S}\rightarrow\overline{\mathbb{H}^3},\rho)$ as follows

\begin{enumerate}
    \item For any lift $\tilde{\gamma_i}$ of $\gamma_i$, map its preferred endpoint to the corresponding endpoint given by $\zeta$.
    \item For any lift of a component of $\lambda_0\setminus\Gamma$, send it to the geodesic in $\mathbb{H}^3$ joining the corresponding preferred endpoints. This is possible since by assumption the preferred endpoints in $\partial_\infty\mathbb{H}^3$ are distinct.
    \item For any lift of an ideal triangle in $\tilde{S}\setminus\tilde{\lambda_0}$, send it to the ideal triangle spanned by the distinct corresponding endpoints.
    \item\label{continuity} Finally, extend continuously to $\tilde{\Gamma}$. This makes the map $\tilde{f}:\tilde{S}\rightarrow\overline{\mathbb{H}^3}$ equivariant by $\rho$.
\end{enumerate}

Note that if every $\rho(\gamma_i)$ was loxodromic and the orientation $\zeta$ agrees with the orientation of $\Gamma$, this will be the classical notion of pleated surface. On the other hand if for some $\rho(\gamma_i)$ the orientations of $\zeta$ and $\Gamma$  disagree, then $\tilde{f}$ is the classical notion of pleated surface for the lamination $\lambda_1$, where $\lambda_1$ is obtained from $\Gamma$ after changing the orientations to have full agreement with $\zeta$.

If $\rho(\gamma_i)$ is elliptic then the situation is a bit more delicate. The continuous extension of step (\ref{continuity}) will send any lift $\tilde{\gamma_i}$ entirely to the preferred endpoint given by $\zeta$. This means that the pull-back metric on $S$ given by $\tilde{f}$ is of finite type with $\ell(\gamma_i)=0$, which is consistent with the fact that the complex length of $\rho(\gamma_i)$ is purely imaginary.

\subsection{Bending tranverse cocycle}\label{subsec:bendingcocycle}

In \cite{Bonahon96}, Bonahon defines the bending transverse cocycle of a pleated surface. This means that for each arc $\alpha$ transverse to a lamination $\lambda$ we have a number $\beta(\alpha) \in \mathbb{R}/2\pi\mathbb{Z}$ called the \textit{bending}, which represents the amount of turning made by the pleated surface between the geodesic faces containing the endpoints of $\alpha$. The cocycle $\beta$ is additive under finite subdivision of a transverse arcs and can be defined as follows. Consider all the geodesics of the lamination that $\alpha$ crosses, and the geodesic faces (or rather plaques, as denoted by Bonahon) of the pleated surface going between those geodesics. The boundary of the plaques define two curves $\eta_{1,2}$ in $\partial_\infty\mathbb{H}^3$. Then the bending of $\alpha$ is defined as the difference of angles between the end plaques minus the integral of the signed curvature of either $\eta_{1,2}$.

For our fixed lamination $(\tilde{S},\tilde{\lambda_0})$ we can finitely decompose any transverse arc into smaller transverse arcs so that each smaller arc intersects $\tilde{\Gamma}$ at most once. This simplifies the description of the bending cocycle as follows.

If an arc $\alpha$ does not intersect $\tilde{\Gamma}$ then it only intersects finitely many geodesic lines of $\tilde{\lambda_0}$. Then the bending is just given by the sum of the finitely many angles involved.

Now say that $\alpha$ intersects exactly one component $\tilde{\gamma_i}$ of $\tilde{\Gamma}$. Furthermore, assume that $\alpha$ only intersects leaves of $\lambda_0$ that accumulate in $\gamma_i$. Then all plaques involved contain the preferred endpoint of $\tilde{\gamma_i}$. This means that one of the two curves $\eta_{1,2}$ degenerates to a point, so we are only left with the angle of the two intersecting end plaques.

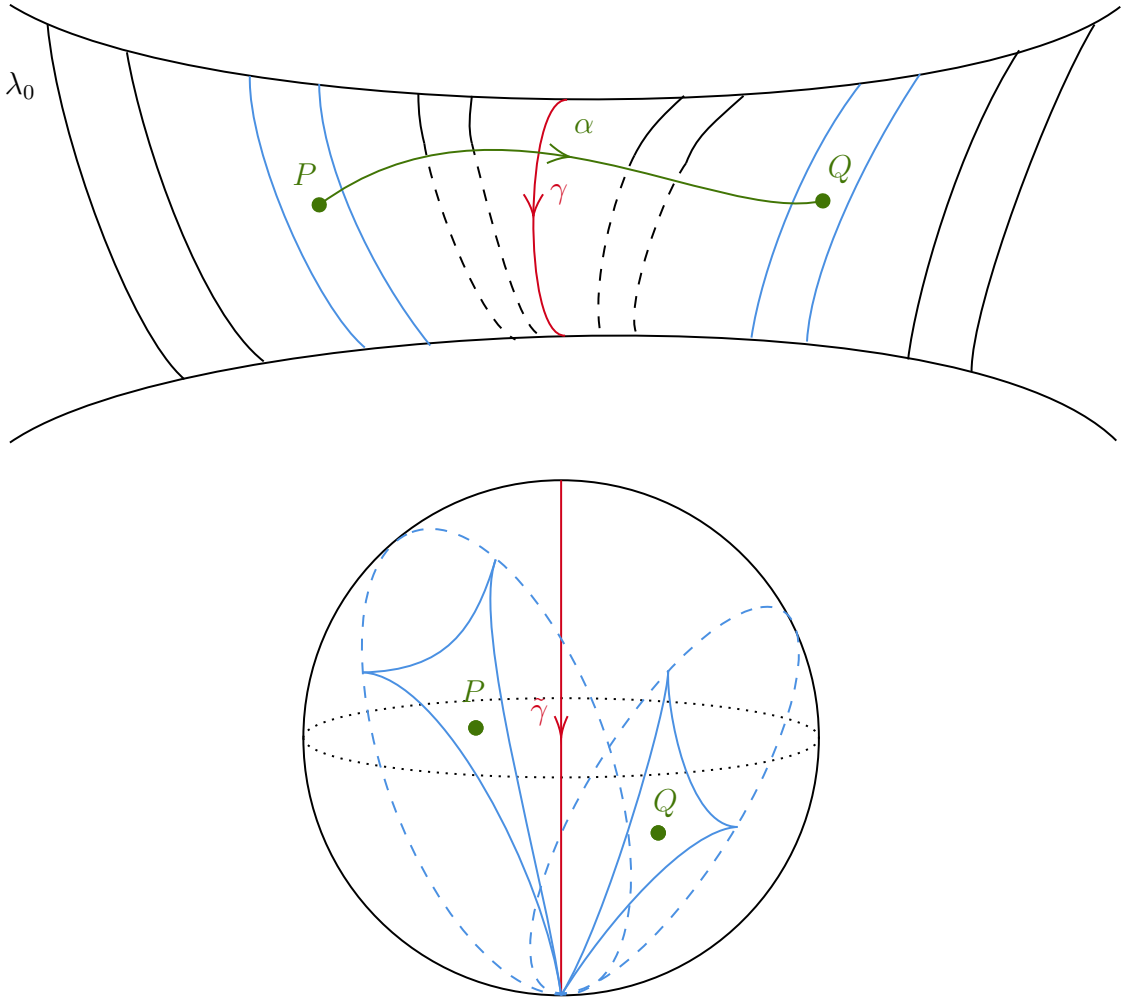
\begin{figure}[hbt!]
    \centering

\tikzset{every picture/.style={line width=0.75pt}} %set default line width to 0.75pt        

\begin{tikzpicture}[x=0.75pt,y=0.75pt,yscale=-1,xscale=1]
%uncomment if require: \path (0,594); %set diagram left start at 0, and has height of 594

%Curve Lines [id:da13251659794988568] 
\draw    (59,254) .. controls (149,192) and (536,173) .. (617,253) ;
%Curve Lines [id:da7966605100218969] 
\draw    (59,33) .. controls (158,95) and (536,98) .. (619,34) ;
%Curve Lines [id:da16214808990405305] 
\draw [color={rgb, 255:red, 74; green, 144; blue, 226 }  ,draw opacity=1 ]   (180,69) .. controls (179,108) and (214,186) .. (238,206) ;
%Curve Lines [id:da5501667776346941] 
\draw [color={rgb, 255:red, 74; green, 144; blue, 226 }  ,draw opacity=1 ]   (215,73) .. controls (215,111) and (241,168) .. (271,205) ;
%Curve Lines [id:da3522199575785483] 
\draw [color={rgb, 255:red, 208; green, 2; blue, 27 }  ,draw opacity=1 ]   (339,200) .. controls (317,201) and (318,80) .. (340,81) ;
\draw [shift={(323.01,139.6)}, rotate = 269.95] [color={rgb, 255:red, 208; green, 2; blue, 27 }  ,draw opacity=1 ][line width=0.75]    (10.93,-4.9) .. controls (6.95,-2.3) and (3.31,-0.67) .. (0,0) .. controls (3.31,0.67) and (6.95,2.3) .. (10.93,4.9)   ;
%Curve Lines [id:da6044964212232327] 
\draw    (78,43) .. controls (83,99) and (120,199) .. (147,222) ;
%Curve Lines [id:da9122004701908486] 
\draw    (118,57) .. controls (126,104) and (156,189) .. (187,213) ;
%Curve Lines [id:da6616578304389573] 
\draw    (265,78) .. controls (265,88) and (266,95) .. (268,103) ;
%Curve Lines [id:da5916299998642276] 
\draw  [dash pattern={on 4.5pt off 4.5pt}]  (268,103) .. controls (272,130) and (296,194) .. (314,202) ;
%Curve Lines [id:da1939117981588876] 
\draw    (292,79) .. controls (291,88) and (290,95) .. (292,103) ;
%Curve Lines [id:da19285614395347817] 
\draw  [dash pattern={on 4.5pt off 4.5pt}]  (292,103) .. controls (298,127) and (313,191) .. (325,199) ;
%Curve Lines [id:da3665180640220027] 
\draw    (606,43) .. controls (583,81) and (544,185) .. (544,218) ;
%Curve Lines [id:da5108989394527021] 
\draw    (568,56) .. controls (537,104) and (516,179) .. (512,212) ;
%Curve Lines [id:da23204373031556003] 
\draw [color={rgb, 255:red, 74; green, 144; blue, 226 }  ,draw opacity=1 ]   (518,68) .. controls (490,110) and (463,167) .. (461,203) ;
%Curve Lines [id:da6626024320967112] 
\draw [color={rgb, 255:red, 74; green, 144; blue, 226 }  ,draw opacity=1 ]   (488,73) .. controls (459,110) and (438,167) .. (433,201) ;
%Curve Lines [id:da07977003169959329] 
\draw    (429,79) .. controls (411,95) and (405,100) .. (400,113) ;
%Curve Lines [id:da7375919823991082] 
\draw  [dash pattern={on 4.5pt off 4.5pt}]  (400,113) .. controls (387,131) and (369,186) .. (375,199) ;
%Curve Lines [id:da9691313880642969] 
\draw    (399,79) .. controls (388,89) and (380,97) .. (374,109) ;
%Curve Lines [id:da8344505509967763] 
\draw  [dash pattern={on 4.5pt off 4.5pt}]  (374,109) .. controls (363,130) and (352,175) .. (357,200) ;
%Curve Lines [id:da5500458902239609] 
\draw [color={rgb, 255:red, 65; green, 117; blue, 5 }  ,draw opacity=1 ]   (215,134) .. controls (305,67) and (414,144) .. (469,132) ;
\draw [shift={(469,132)}, rotate = 347.69] [color={rgb, 255:red, 65; green, 117; blue, 5 }  ,draw opacity=1 ][fill={rgb, 255:red, 65; green, 117; blue, 5 }  ,fill opacity=1 ][line width=0.75]      (0, 0) circle [x radius= 3.35, y radius= 3.35]   ;
\draw [shift={(341.39,109.8)}, rotate = 188.83] [color={rgb, 255:red, 65; green, 117; blue, 5 }  ,draw opacity=1 ][line width=0.75]    (10.93,-4.9) .. controls (6.95,-2.3) and (3.31,-0.67) .. (0,0) .. controls (3.31,0.67) and (6.95,2.3) .. (10.93,4.9)   ;
\draw [shift={(215,134)}, rotate = 323.33] [color={rgb, 255:red, 65; green, 117; blue, 5 }  ,draw opacity=1 ][fill={rgb, 255:red, 65; green, 117; blue, 5 }  ,fill opacity=1 ][line width=0.75]      (0, 0) circle [x radius= 3.35, y radius= 3.35]   ;
%Shape: Circle [id:dp5369214319819402] 
\draw   (207,403) .. controls (207,331.2) and (265.2,273) .. (337,273) .. controls (408.8,273) and (467,331.2) .. (467,403) .. controls (467,474.8) and (408.8,533) .. (337,533) .. controls (265.2,533) and (207,474.8) .. (207,403) -- cycle ;
%Shape: Ellipse [id:dp5560440668119206] 
\draw  [dash pattern={on 0.84pt off 2.51pt}] (207,403) .. controls (207,391.95) and (265.2,383) .. (337,383) .. controls (408.8,383) and (467,391.95) .. (467,403) .. controls (467,414.05) and (408.8,423) .. (337,423) .. controls (265.2,423) and (207,414.05) .. (207,403) -- cycle ;
%Straight Lines [id:da812697262942959] 
\draw [color={rgb, 255:red, 208; green, 2; blue, 27 }  ,draw opacity=1 ]   (337,273) -- (337,533) ;
\draw [shift={(337,403)}, rotate = 270] [color={rgb, 255:red, 208; green, 2; blue, 27 }  ,draw opacity=1 ][line width=0.75]    (10.93,-3.29) .. controls (6.95,-1.4) and (3.31,-0.3) .. (0,0) .. controls (3.31,0.3) and (6.95,1.4) .. (10.93,3.29)   ;
%Curve Lines [id:da667222900434215] 
\draw [color={rgb, 255:red, 74; green, 144; blue, 226 }  ,draw opacity=1 ]   (237,370) .. controls (277,372) and (329,473) .. (337,533) ;
%Curve Lines [id:da4548425381465906] 
\draw [color={rgb, 255:red, 74; green, 144; blue, 226 }  ,draw opacity=1 ]   (304,313) .. controls (291,348) and (329,473) .. (337,533) ;
%Curve Lines [id:da10305772230372079] 
\draw [color={rgb, 255:red, 74; green, 144; blue, 226 }  ,draw opacity=1 ]   (237,370) .. controls (279,369) and (293,349) .. (304,313) ;
%Curve Lines [id:da03213091447845984] 
\draw [color={rgb, 255:red, 74; green, 144; blue, 226 }  ,draw opacity=1 ]   (391,369) .. controls (391,388) and (360,495) .. (337,533) ;
%Curve Lines [id:da8091594160774163] 
\draw [color={rgb, 255:red, 74; green, 144; blue, 226 }  ,draw opacity=1 ]   (391,369) .. controls (390,389) and (401,448) .. (426,448) ;
%Curve Lines [id:da018916389579220372] 
\draw [color={rgb, 255:red, 74; green, 144; blue, 226 }  ,draw opacity=1 ]   (337,533) .. controls (360,496) and (401,449) .. (426,448) ;
%Shape: Ellipse [id:dp7466095649714939] 
\draw  [color={rgb, 255:red, 74; green, 144; blue, 226 }  ,draw opacity=1 ][dash pattern={on 4.5pt off 4.5pt}] (261.76,299.19) .. controls (290.93,288.46) and (333.67,331.66) .. (357.21,395.67) .. controls (380.75,459.68) and (376.18,520.27) .. (347,531) .. controls (317.82,541.73) and (275.09,498.53) .. (251.55,434.52) .. controls (228.01,370.51) and (232.58,309.92) .. (261.76,299.19) -- cycle ;
%Shape: Ellipse [id:dp4725083432925359] 
\draw  [color={rgb, 255:red, 74; green, 144; blue, 226 }  ,draw opacity=1 ][dash pattern={on 4.5pt off 4.5pt}] (437.53,337.15) .. controls (463.75,333.2) and (463.34,373.54) .. (436.62,427.25) .. controls (409.9,480.96) and (366.99,527.7) .. (340.77,531.65) .. controls (314.56,535.6) and (314.96,495.26) .. (341.68,441.56) .. controls (368.4,387.85) and (411.32,341.1) .. (437.53,337.15) -- cycle ;
%Straight Lines [id:da13269576623333457] 
\draw [color={rgb, 255:red, 65; green, 117; blue, 5 }  ,draw opacity=1 ]   (294,398) ;
\draw [shift={(294,398)}, rotate = 0] [color={rgb, 255:red, 65; green, 117; blue, 5 }  ,draw opacity=1 ][fill={rgb, 255:red, 65; green, 117; blue, 5 }  ,fill opacity=1 ][line width=0.75]      (0, 0) circle [x radius= 3.35, y radius= 3.35]   ;
\draw [shift={(294,398)}, rotate = 0] [color={rgb, 255:red, 65; green, 117; blue, 5 }  ,draw opacity=1 ][fill={rgb, 255:red, 65; green, 117; blue, 5 }  ,fill opacity=1 ][line width=0.75]      (0, 0) circle [x radius= 3.35, y radius= 3.35]   ;
%Straight Lines [id:da807449788967342] 
\draw [color={rgb, 255:red, 65; green, 117; blue, 5 }  ,draw opacity=1 ]   (386,451) ;
\draw [shift={(386,451)}, rotate = 0] [color={rgb, 255:red, 65; green, 117; blue, 5 }  ,draw opacity=1 ][fill={rgb, 255:red, 65; green, 117; blue, 5 }  ,fill opacity=1 ][line width=0.75]      (0, 0) circle [x radius= 3.35, y radius= 3.35]   ;
\draw [shift={(386,451)}, rotate = 0] [color={rgb, 255:red, 65; green, 117; blue, 5 }  ,draw opacity=1 ][fill={rgb, 255:red, 65; green, 117; blue, 5 }  ,fill opacity=1 ][line width=0.75]      (0, 0) circle [x radius= 3.35, y radius= 3.35]   ;

% Text Node
\draw (330,122.4) node [anchor=north west][inner sep=0.75pt]  [color={rgb, 255:red, 208; green, 2; blue, 27 }  ,opacity=1 ]  {$\gamma $};
% Text Node
\draw (55,65.4) node [anchor=north west][inner sep=0.75pt]    {$\lambda _{0}$};
% Text Node
\draw (342,89.4) node [anchor=north west][inner sep=0.75pt]  [color={rgb, 255:red, 65; green, 117; blue, 5 }  ,opacity=1 ]  {$\alpha $};
% Text Node
\draw (320,381.4) node [anchor=north west][inner sep=0.75pt]  [color={rgb, 255:red, 208; green, 2; blue, 27 }  ,opacity=1 ]  {$\tilde{\gamma }$};
% Text Node
\draw (200,111.4) node [anchor=north west][inner sep=0.75pt]  [color={rgb, 255:red, 65; green, 117; blue, 5 }  ,opacity=1 ]  {$P$};
% Text Node
\draw (470,107.4) node [anchor=north west][inner sep=0.75pt]  [color={rgb, 255:red, 65; green, 117; blue, 5 }  ,opacity=1 ]  {$Q$};
% Text Node
\draw (285,372.4) node [anchor=north west][inner sep=0.75pt]  [color={rgb, 255:red, 65; green, 117; blue, 5 }  ,opacity=1 ]  {$P$};
% Text Node
\draw (382,427.4) node [anchor=north west][inner sep=0.75pt]  [color={rgb, 255:red, 65; green, 117; blue, 5 }  ,opacity=1 ]  {$Q$};

\end{tikzpicture}
    
    \caption{Bending of a transverse curve $\alpha$ intersecting a unique component $\gamma$ of $\Gamma$.}
    \label{fig:bending}
\end{figure}\noindent

Now it is easy to see that for our definition of (generalized) pleated surface, the two above definitions are well-defined and are additive under finite subdivision. Then for a general transverse arc $\alpha$ we can take any finite subdivision so every arc is of one of the two cases analyzed above, and define the bending of $\alpha$ as the sum of bending. Since the cocycle was already additive between the two types, this shows that the bending cocycle is well-defined and additive under finite subdivision. Moreover, this implies that for a smooth 1-parameter family of $\Gamma$-adapted representations $\rho_t$ where the orientation/endpoint choice $\zeta$ varies continuously, the bending varies smoothly. This is because the family of preferred endpoints will vary smoothly, and the bending of any arc is decomposed as the finite sum of finitely many angles of planes defined by preferred endpoints.

\subsection{Volume variation}\label{subsec:volumevariation}

Let $\rho_t$ be a smooth 1-parameter family of $\Gamma$-adapted representations. Let $f_t$ be the $\lambda_0$ generalized pleated surface, $m_t$ the metric induced from $\mathbb{H}^3$ in $S$ by $f_t$, and let $V_t$ be the volume of a $3$-chain bounded by $f_t$. Then
\[V'_t = \frac12 \ell_t(b'_t)
\]
where $b_t$ is the bending cocycle of $f_t$, and $\ell_t$ is the length of the (real-valued) transverse cocycle $b'_t$ with respect to the induced metric $m_t$. This is known for classical pleated surfaces as the Bonahon-Schalfli formula, as proved by Bonahon in \cite{Bonahon98}. We will explain the outline of Bonahon's proof while justifying that the method holds for our case. This in particular means that we need to explain how to make sense of the formula when elliptic transformations are involved.

\begin{figure}[hbt!]
    \centering

\tikzset{every picture/.style={line width=0.75pt}} %set default line width to 0.75pt        

\begin{tikzpicture}[x=0.75pt,y=0.75pt,yscale=-1,xscale=1]

%Curve Lines [id:da8205885015738903] 
\draw    (51,261) .. controls (141,199) and (528,180) .. (609,260) ;
%Curve Lines [id:da3619675145739716] 
\draw    (50,40) .. controls (149,102) and (527,105) .. (610,41) ;
%Curve Lines [id:da36139223829290246] 
\draw [color={rgb, 255:red, 208; green, 2; blue, 27 }  ,draw opacity=1 ]   (331,207) .. controls (309,208) and (310,87) .. (332,88) ;
\draw [shift={(315.01,146.6)}, rotate = 269.95] [color={rgb, 255:red, 208; green, 2; blue, 27 }  ,draw opacity=1 ][line width=0.75]    (10.93,-4.9) .. controls (6.95,-2.3) and (3.31,-0.67) .. (0,0) .. controls (3.31,0.67) and (6.95,2.3) .. (10.93,4.9)   ;
%Curve Lines [id:da41288118512466343] 
\draw    (70,50) .. controls (75,106) and (112,206) .. (139,229) ;
%Curve Lines [id:da7718748975006098] 
\draw    (110,64) .. controls (118,111) and (148,196) .. (179,220) ;
%Curve Lines [id:da6664395363488687] 
\draw    (171,76) .. controls (170,115) and (205,193) .. (229,213) ;
%Curve Lines [id:da7491030473163383] 
\draw    (206,80) .. controls (206,118) and (232,175) .. (262,212) ;
%Curve Lines [id:da04310382825166714] 
\draw    (257,85) .. controls (257,95) and (258,102) .. (260,110) ;
%Curve Lines [id:da7169514909757451] 
\draw  [dash pattern={on 4.5pt off 4.5pt}]  (260,110) .. controls (264,137) and (288,201) .. (306,209) ;
%Curve Lines [id:da07140811475513709] 
\draw    (284,86) .. controls (283,95) and (282,102) .. (284,110) ;
%Curve Lines [id:da9989636365529082] 
\draw  [dash pattern={on 4.5pt off 4.5pt}]  (284,110) .. controls (290,134) and (305,198) .. (317,206) ;
%Curve Lines [id:da35314020123913004] 
\draw    (599,50) .. controls (576,88) and (536,191) .. (533,225) ;
%Curve Lines [id:da8669531904952792] 
\draw    (560,63) .. controls (529,111) and (498,179) .. (496,215) ;
%Curve Lines [id:da28707525828623637] 
\draw    (510,75) .. controls (482,117) and (446,176) .. (449,212) ;
%Curve Lines [id:da062379671008216286] 
\draw    (480,80) .. controls (442,116) and (420,177) .. (420,208) ;
%Curve Lines [id:da8746026276495005] 
\draw    (421,86) .. controls (403,102) and (397,107) .. (392,120) ;
%Curve Lines [id:da3299146315126307] 
\draw  [dash pattern={on 4.5pt off 4.5pt}]  (392,120) .. controls (379,138) and (361,193) .. (367,206) ;
%Curve Lines [id:da9384751895833081] 
\draw    (391,86) .. controls (380,96) and (372,104) .. (366,116) ;
%Curve Lines [id:da0840253465728189] 
\draw  [dash pattern={on 4.5pt off 4.5pt}]  (366,116) .. controls (355,137) and (344,182) .. (349,207) ;
%Curve Lines [id:da8150293345959787] 
\draw [color={rgb, 255:red, 74; green, 144; blue, 226 }  ,draw opacity=1 ]   (197,120) .. controls (232,122) and (423,121) .. (465,121) ;
%Shape: Polygon Curved [id:ds4528560693610355] 
\draw  [color={rgb, 255:red, 74; green, 144; blue, 226 }  ,draw opacity=1 ][fill={rgb, 255:red, 206; green, 232; blue, 250 }  ,fill opacity=0.45 ] (494,77) .. controls (478,95) and (474,99) .. (465,121) .. controls (497,107) and (539,86) .. (575,66) .. controls (551,101) and (523,172) .. (518,222) .. controls (425,197) and (222,208) .. (153,226) .. controls (129,197) and (90,105) .. (88,70) .. controls (115,87) and (160,105) .. (197,120) .. controls (192,108) and (187,92) .. (186,78) .. controls (271,91) and (443,90) .. (494,77) -- cycle ;

% Text Node
\draw (322,129.4) node [anchor=north west][inner sep=0.75pt]  [color={rgb, 255:red, 208; green, 2; blue, 27 }  ,opacity=1 ]  {$\gamma $};
% Text Node
\draw (337,58.4) node [anchor=north west][inner sep=0.75pt]  [color={rgb, 255:red, 74; green, 144; blue, 226 }  ,opacity=1 ]  {$R$};
% Text Node
\draw (47,72.4) node [anchor=north west][inner sep=0.75pt]    {$\lambda _{0}$};

\end{tikzpicture}
    
    \caption{Rectangle $R$ going around the closed geodesic $\gamma$.}
    \label{fig:rectangles1}
\end{figure}
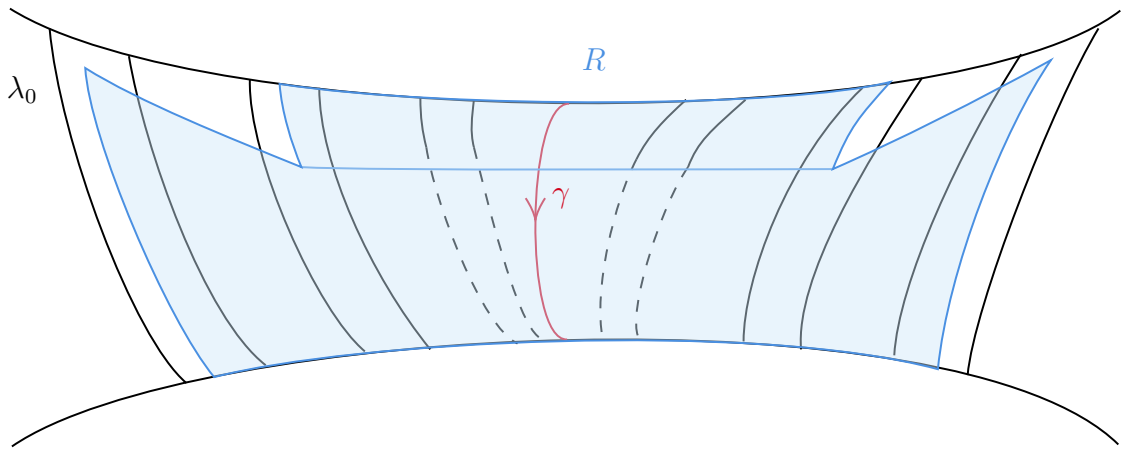\noindent

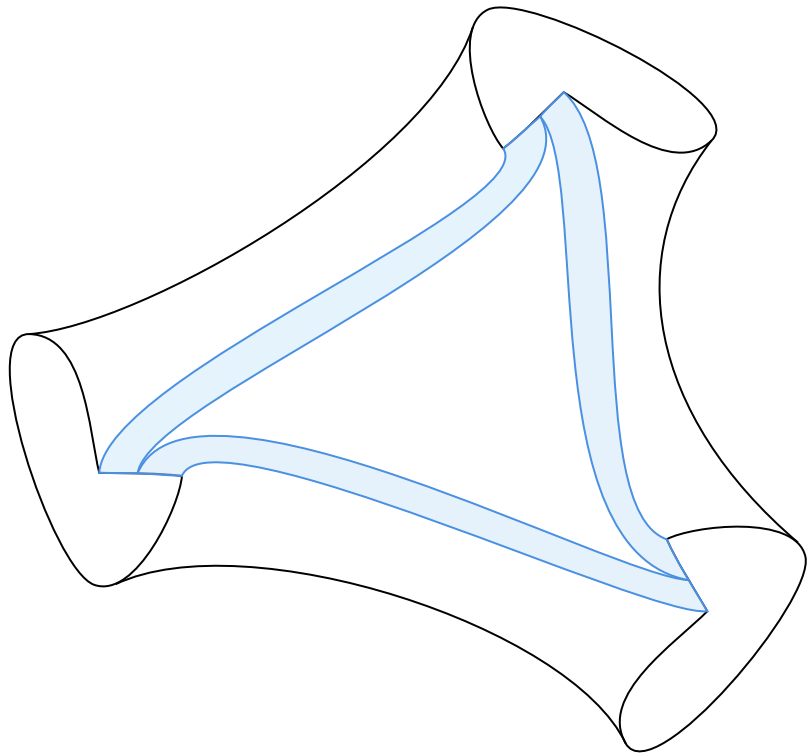
\begin{figure}[hbt!]
    \centering

\tikzset{every picture/.style={line width=0.75pt}} %set default line width to 0.75pt        

\begin{tikzpicture}[x=0.75pt,y=0.75pt,yscale=-1,xscale=1]
%uncomment if require: \path (0,520); %set diagram left start at 0, and has height of 520

%Shape: Regular Polygon [id:dp025163981171763217] 
\draw   (348.63,75.04) .. controls (368.81,65.4) and (481.82,121.44) .. (461.46,141.08) .. controls (441.11,160.72) and (405.69,128.07) .. (386.89,116.73) .. controls (373.66,129.5) and (367.53,136.39) .. (356.37,145.19) .. controls (346.57,134.01) and (328.45,84.68) .. (348.63,75.04) -- cycle ;
%Shape: Regular Polygon [id:dp7964955779216303] 
\draw   (150.06,365.36) .. controls (129,357.86) and (88.98,238.24) .. (117.26,238.81) .. controls (145.54,239.38) and (147.4,287.51) .. (152.62,308.84) .. controls (171,309.21) and (180.21,308.69) .. (194.32,310.39) .. controls (193.31,325.22) and (171.13,372.86) .. (150.06,365.36) -- cycle ;
%Shape: Regular Polygon [id:dp4586957096246733] 
\draw   (508.79,351.26) .. controls (513.44,373.13) and (432.52,469.89) .. (418.19,445.51) .. controls (403.86,421.12) and (443.88,394.32) .. (459.3,378.7) .. controls (449.99,362.85) and (444.72,355.28) .. (438.78,342.37) .. controls (451.94,335.46) and (504.14,329.39) .. (508.79,351.26) -- cycle ;
%Curve Lines [id:da36864875737860237] 
\draw    (117.26,238.81) .. controls (173,235) and (319,152) .. (341,84) ;
%Curve Lines [id:da9855022186728029] 
\draw    (461.46,141.08) .. controls (416,199) and (429,279) .. (505,344) ;
%Curve Lines [id:da9193724140356045] 
\draw    (161,365) .. controls (227,334) and (390,386) .. (418.19,445.51) ;
%Shape: Polygon Curved [id:ds4981583361340991] 
\draw  [color={rgb, 255:red, 74; green, 144; blue, 226 }  ,draw opacity=1 ][fill={rgb, 255:red, 206; green, 232; blue, 250 }  ,fill opacity=0.53 ] (375,129) .. controls (379,124) and (384,120) .. (386.89,116.73) .. controls (426,151) and (395,327) .. (438.78,342.37) .. controls (442,349) and (445,355) .. (450,363) .. controls (370,349) and (401,158) .. (375,129) -- cycle ;
%Shape: Polygon Curved [id:ds2742621711664295] 
\draw  [color={rgb, 255:red, 74; green, 144; blue, 226 }  ,draw opacity=1 ][fill={rgb, 255:red, 206; green, 232; blue, 250 }  ,fill opacity=0.53 ] (172,309) .. controls (199,246) and (411,361) .. (450,363) .. controls (454,370) and (455.3,371.7) .. (459.3,378.7) .. controls (429,382) and (209,276) .. (194.32,310.39) .. controls (189,310) and (185,310) .. (172,309) -- cycle ;
%Shape: Polygon Curved [id:ds39531947147305724] 
\draw  [color={rgb, 255:red, 74; green, 144; blue, 226 }  ,draw opacity=1 ][fill={rgb, 255:red, 206; green, 232; blue, 250 }  ,fill opacity=0.52 ] (356.37,145.19) .. controls (364,139) and (369,134) .. (375,129) .. controls (407,185) and (178,270) .. (172,309) .. controls (164,309) and (158,309) .. (152.62,308.84) .. controls (156,259) and (375,174) .. (356.37,145.19) -- cycle ;

\end{tikzpicture}

    \caption{Rectangles on the complement of the cuffs}
    \label{fig:rectangles2}
\end{figure}\noindent

Cover $\lambda_0$ by geodesic rectangles $R^0_1,\ldots R^0_m$ with disjoint interior, so that the components of $\lambda_0\cap R^0_i$ are parallel to opposite sides of the rectangle. Each rectangle $R^0_i$ can be collapsed to an edge in order to obtain an embedded graph in $S$, and we can complete that graph to a triangulation $T_0$ of $S$, which lifts to a triangulation $\tilde{T_0}$ of $\tilde{S}$. Take then a $\rho_0$ equivariant map $g:\tilde{S}\rightarrow \mathbb{H}^3$ that is polyhedral with respect to $\tilde{T_0}$ and homotopic to $f_0$ through $\rho_0$-equivariant maps. For $t$ small, take $m_t$-geodesic rectangles $R^t_i$ so that the analogous statement holds for $\lambda_t$, $R^t_i$ is smooth on $t$, and up to isotopy on $S$ we have that $T_t=T_0$. Consider then $\rho_t$-equivariant maps $g_t:\tilde{S}\rightarrow \mathbb{H}^3$ that are polyhedral with respect to $\tilde{T_0}$ and homotopic to $f_t$ through $\rho_t$-equivariant maps. By the Schl\"afli Formula for polyhedral maps, we can calculate the variation of volume for the quotient of a $3$-chain bounded by $g_t$. Hence we can reduce the problem to calculate the variation of volume for the quotient of a $\rho_t$-equivariant homotopy between $f_t,g_t$. Since we can take maps $h_t:S\rightarrow S$ homotopic to the identity so that $h_t(R^t_i)$ is an arc of $T$, the volume of the quotient of the homotopy between $g_t$ and $g_t\circ\tilde{h_t}$ is equal to $0$. Then we can further reduce to calculate the variation of volume for the quotient of a $\rho_t$-equivariant homotopy $H_t$ between $f_t,g_t\circ\tilde{h_t}$.

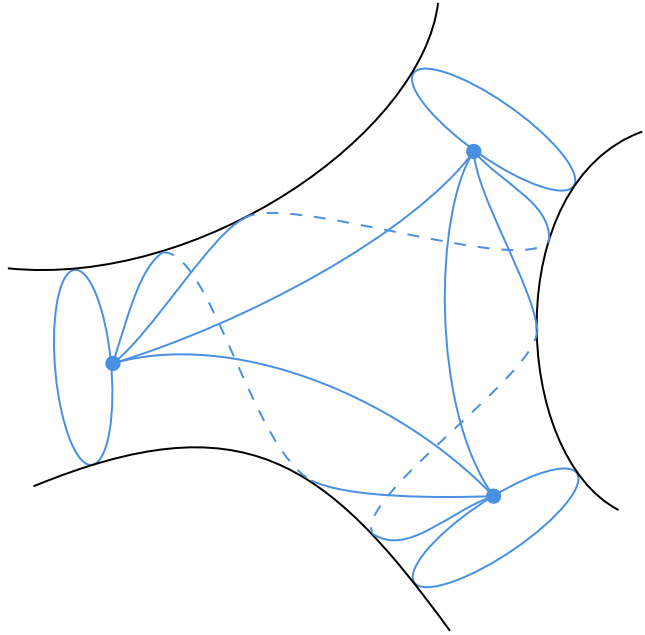
\begin{figure}[hbt!]
    \centering

\tikzset{every picture/.style={line width=0.75pt}} %set default line width to 0.75pt        

\begin{tikzpicture}[x=0.75pt,y=0.75pt,yscale=-1,xscale=1]
%uncomment if require: \path (0,493); %set diagram left start at 0, and has height of 493

%Shape: Ellipse [id:dp03286877471877192] 
\draw  [color={rgb, 255:red, 74; green, 144; blue, 226 }  ,draw opacity=1 ] (374.69,101.25) .. controls (361.31,85.32) and (365.37,77.22) .. (383.77,83.16) .. controls (402.16,89.1) and (427.92,106.82) .. (441.31,122.75) .. controls (454.69,138.68) and (450.63,146.78) .. (432.23,140.84) .. controls (413.84,134.9) and (388.08,117.18) .. (374.69,101.25) -- cycle ;
%Shape: Ellipse [id:dp6745289230429234] 
\draw  [color={rgb, 255:red, 74; green, 144; blue, 226 }  ,draw opacity=1 ] (187.74,199.61) .. controls (191.28,179.11) and (200.09,176.99) .. (207.41,194.88) .. controls (214.74,212.77) and (217.8,243.89) .. (214.26,264.39) .. controls (210.72,284.89) and (201.91,287.01) .. (194.59,269.12) .. controls (187.26,251.23) and (184.2,220.11) .. (187.74,199.61) -- cycle ;
%Shape: Ellipse [id:dp7364177996065622] 
\draw  [color={rgb, 255:red, 74; green, 144; blue, 226 }  ,draw opacity=1 ] (430.93,285.72) .. controls (450.58,278.9) and (456.7,285.58) .. (444.59,300.65) .. controls (432.48,315.72) and (406.73,333.46) .. (387.07,340.28) .. controls (367.42,347.1) and (361.3,340.42) .. (373.41,325.35) .. controls (385.52,310.28) and (411.27,292.54) .. (430.93,285.72) -- cycle ;
%Curve Lines [id:da47362268986372547] 
\draw    (163,182) .. controls (260,192) and (372,114) .. (380,48) ;
%Curve Lines [id:da044423295851246936] 
\draw    (471,304) .. controls (417,276) and (411,139) .. (483,113) ;
%Curve Lines [id:da15781138172822828] 
\draw    (176,292) .. controls (281,249) and (323,278) .. (386,365) ;
%Curve Lines [id:da20453415930712127] 
\draw [color={rgb, 255:red, 74; green, 144; blue, 226 }  ,draw opacity=1 ]   (216,230) .. controls (282,210) and (367,165) .. (398,123) ;
\draw [shift={(398,123)}, rotate = 306.43] [color={rgb, 255:red, 74; green, 144; blue, 226 }  ,draw opacity=1 ][fill={rgb, 255:red, 74; green, 144; blue, 226 }  ,fill opacity=1 ][line width=0.75]      (0, 0) circle [x radius= 3.35, y radius= 3.35]   ;
\draw [shift={(216,230)}, rotate = 343.14] [color={rgb, 255:red, 74; green, 144; blue, 226 }  ,draw opacity=1 ][fill={rgb, 255:red, 74; green, 144; blue, 226 }  ,fill opacity=1 ][line width=0.75]      (0, 0) circle [x radius= 3.35, y radius= 3.35]   ;
%Curve Lines [id:da776970382882233] 
\draw [color={rgb, 255:red, 74; green, 144; blue, 226 }  ,draw opacity=1 ]   (398,123) .. controls (378,153) and (376,253) .. (408,297) ;
\draw [shift={(408,297)}, rotate = 53.97] [color={rgb, 255:red, 74; green, 144; blue, 226 }  ,draw opacity=1 ][fill={rgb, 255:red, 74; green, 144; blue, 226 }  ,fill opacity=1 ][line width=0.75]      (0, 0) circle [x radius= 3.35, y radius= 3.35]   ;
%Curve Lines [id:da006815777798002243] 
\draw [color={rgb, 255:red, 74; green, 144; blue, 226 }  ,draw opacity=1 ]   (216,230) .. controls (275,213) and (358,245) .. (408,297) ;
%Curve Lines [id:da08904795248234021] 
\draw [color={rgb, 255:red, 74; green, 144; blue, 226 }  ,draw opacity=1 ]   (216,230) .. controls (235.29,215.53) and (266,164) .. (282,157) ;
%Curve Lines [id:da6256262051139252] 
\draw [color={rgb, 255:red, 74; green, 144; blue, 226 }  ,draw opacity=1 ] [dash pattern={on 4.5pt off 4.5pt}]  (282,157) .. controls (301.29,142.53) and (423,187) .. (436,168) ;
%Curve Lines [id:da5000358157767812] 
\draw [color={rgb, 255:red, 74; green, 144; blue, 226 }  ,draw opacity=1 ]   (398,123) .. controls (413,140) and (437,149) .. (436,168) ;
%Curve Lines [id:da8517281462062047] 
\draw [color={rgb, 255:red, 74; green, 144; blue, 226 }  ,draw opacity=1 ]   (398,123) .. controls (399,145) and (428,191) .. (430,213) ;
%Curve Lines [id:da9058808775009348] 
\draw [color={rgb, 255:red, 74; green, 144; blue, 226 }  ,draw opacity=1 ] [dash pattern={on 4.5pt off 4.5pt}]  (430,213) .. controls (431,235) and (337,298) .. (347,316) ;
%Curve Lines [id:da5581227074883772] 
\draw [color={rgb, 255:red, 74; green, 144; blue, 226 }  ,draw opacity=1 ]   (347,316) .. controls (365,328) and (386,304) .. (408,297) ;
%Curve Lines [id:da09966885580137008] 
\draw [color={rgb, 255:red, 74; green, 144; blue, 226 }  ,draw opacity=1 ]   (241,174) .. controls (230,182) and (225,201) .. (216,230) ;
%Curve Lines [id:da8587652757660786] 
\draw [color={rgb, 255:red, 74; green, 144; blue, 226 }  ,draw opacity=1 ] [dash pattern={on 4.5pt off 4.5pt}]  (241,174) .. controls (265,172) and (286,273) .. (316,289) ;
%Curve Lines [id:da5181852088716932] 
\draw [color={rgb, 255:red, 74; green, 144; blue, 226 }  ,draw opacity=1 ]   (316,289) .. controls (334,298) and (380,298) .. (408,297) ;

\end{tikzpicture}

    \caption{Triangulation on a pair of pants}
    \label{fig:graph}
\end{figure}\noindent

The next step is to divide the homotopy $H_t$ into a family of $\rho_t$ equivariant polyhedral pieces. Fixing a rectangle $R^t_i$, and given $g_t\circ\tilde{h_t}=H_t(.,0)$ sends $R^t_i$ to a geodesic segment, we want extend $H$ to $R^t_i\times[0,1]$ by geodesic segments so that $f_t=H_t(.,1)$.  In order to do so, for each component $R$ of $R^t_i\setminus\lambda_t$, we define $H_t(R\times[0,1])$ so that decomposes into the union of a pyramid with square basis given by $R$, and a tetrahedra that shares a side with the pyramid (see Figure \ref{fig:3chain}). Because $f_t, g_t\circ\tilde{h_t}$ are $\rho_t$-equivariant, this family of polyhedra $P_t$ is $\rho_t$-equivariant. Since $(\lambda_t\cap R^t_i) \times [0,1]$ has $3$-dimensional Lebesgue measure $0$, we can focus solely in the family of polyhedra $P_t$ in order to calculate volumes.

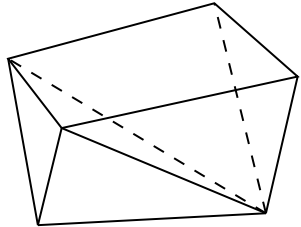
\begin{figure}[hbt!]
    \centering

\tikzset{every picture/.style={line width=0.75pt}} %set default line width to 0.75pt        

\begin{tikzpicture}[x=0.75pt,y=0.75pt,yscale=-1,xscale=1]

%Shape: Polygon [id:ds9435713536124832] 
\draw   (356,78) -- (398,115) -- (279,141) -- (252,106) -- cycle ;
%Straight Lines [id:da34323503010546785] 
\draw    (279,141) -- (382,184) ;
%Straight Lines [id:da9634074418960312] 
\draw    (398,115) -- (382,184) ;
%Straight Lines [id:da8545105394660162] 
\draw    (382,184) -- (267,190) ;
%Straight Lines [id:da21464502889563986] 
\draw    (279,141) -- (267,190) ;
%Straight Lines [id:da13871341490572497] 
\draw    (252,106) -- (267,190) ;
%Straight Lines [id:da31849801903143504] 
\draw  [dash pattern={on 4.5pt off 4.5pt}]  (382,184) -- (252,106) ;
%Straight Lines [id:da05485951367355879] 
\draw  [dash pattern={on 4.5pt off 4.5pt}]  (382,184) -- (356,78) ;

\end{tikzpicture}

    \caption{$3$-chain obtained as the union of a pyramid and a prism}
    \label{fig:3chain}
\end{figure}\noindent

Recall that the variation of volume of a polyhedra is given by the sum of half lengths of edges times the variation of the dihedral angle. Then Bonahon \cite{Bonahon98} argues that for any interior edge and for any edge share with $g_t\circ\tilde{h_t}$, the different contributions cancel out. As for edges appearing in $f_t$, their sum can be reinterpreted as the half-length of the variation of the bending cocycle. This is the delicate part of the argument, because as Bonahon points out, edges are not locally finite, so appropriate sumability and convergence should be proved. Bonahon's statement covers the case when all $\gamma\in\Gamma$ are loxodromic, so we are left to justify when some $\gamma$ is elliptic. Hence we concentrate on this case.

For $\lambda_0$ the polyhedral subdivision can be simplified in such a way that the questions of sumability and convergence are easier to conclude.

\begin{figure}[hbt!]
    \centering

\tikzset{every picture/.style={line width=0.75pt}} %set default line width to 0.75pt        

\begin{tikzpicture}[x=0.75pt,y=0.75pt,yscale=-1,xscale=1]

%Shape: Polygon [id:ds2981801918319684] 
\draw   (208,72) -- (235,119) -- (126,142) -- (101,96) -- cycle ;
%Straight Lines [id:da15156308453576006] 
\draw [color={rgb, 255:red, 208; green, 2; blue, 27 }  ,draw opacity=1 ]   (213,182) -- (109,201) ;
%Straight Lines [id:da9415743876526359] 
\draw    (126,142) -- (109,201) ;
%Straight Lines [id:da0868089727289898] 
\draw    (235,119) -- (213,182) ;
%Straight Lines [id:da7011883399846026] 
\draw    (109,201) -- (101,96) ;
%Straight Lines [id:da3955853256459261] 
\draw  [dash pattern={on 4.5pt off 4.5pt}]  (213,182) -- (208,72) ;
%Shape: Polygon [id:ds4558239033883633] 
\draw   (488,75) -- (515,122) -- (406,145) -- (381,99) -- cycle ;
%Straight Lines [id:da8532691258835086] 
\draw    (406,145) -- (389,204) ;
%Straight Lines [id:da08052245853513496] 
\draw    (515,122) -- (389,204) ;
%Straight Lines [id:da6164706140473661] 
\draw    (389,204) -- (381,99) ;
%Straight Lines [id:da3185200508575663] 
\draw  [dash pattern={on 4.5pt off 4.5pt}]  (389,204) -- (488,75) ;
%Straight Lines [id:da9017339769794814] 
\draw [color={rgb, 255:red, 208; green, 2; blue, 27 }  ,draw opacity=1 ]   (493,185) -- (389,204) ;

\end{tikzpicture}

    \caption{Prisms obtained as the $3$-chain in the cases when $\gamma$ is loxodromic (left) or elliptic (right).}
    \label{fig:prisms}
\end{figure}
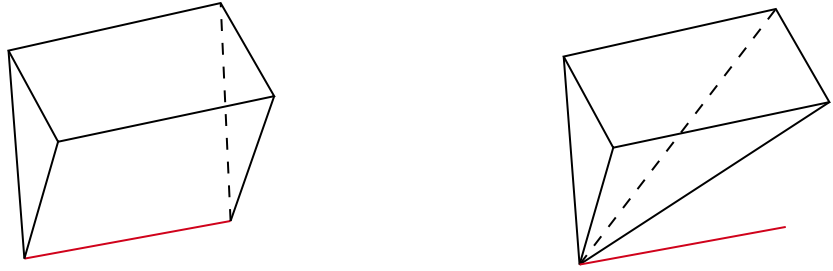\noindent

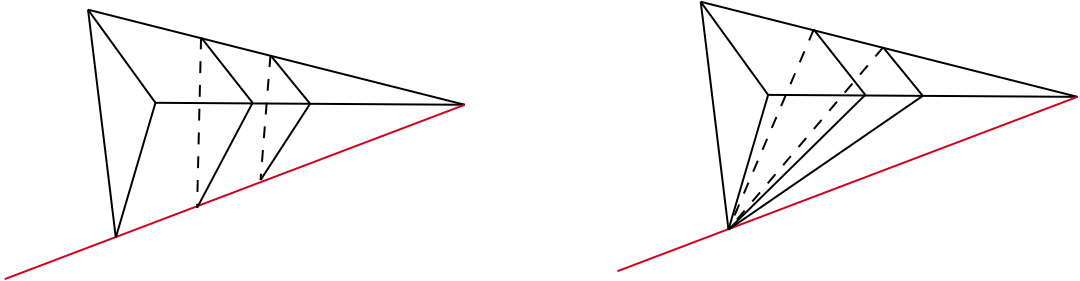
\begin{figure}[hbt!]
    \centering

\tikzset{every picture/.style={line width=0.75pt}} %set default line width to 0.75pt        

\begin{tikzpicture}[x=0.75pt,y=0.75pt,yscale=-1,xscale=1]

%Straight Lines [id:da9063511662846582] 
\draw    (95,51) -- (285,99) ;
%Straight Lines [id:da6202572247214533] 
\draw    (285,99) -- (129,98) ;
%Straight Lines [id:da9014757801225564] 
\draw [color={rgb, 255:red, 208; green, 2; blue, 27 }  ,draw opacity=1 ]   (53,187) -- (285,99) ;
%Straight Lines [id:da68660755594133] 
\draw    (95,51) -- (129,98) ;
%Straight Lines [id:da07558051710719016] 
\draw    (129,98) -- (109,166) ;
%Straight Lines [id:da798497853317514] 
\draw    (95,51) -- (109,166) ;
%Straight Lines [id:da110128114509215] 
\draw    (152,65) -- (178,98) ;
%Straight Lines [id:da20529918233310984] 
\draw    (178,98) -- (150,151) ;
%Straight Lines [id:da46303464782995674] 
\draw  [dash pattern={on 4.5pt off 4.5pt}]  (152,65) -- (150,151) ;
%Straight Lines [id:da6699681955982686] 
\draw    (187,74) -- (207,98.5) ;
%Straight Lines [id:da787660167054977] 
\draw  [dash pattern={on 4.5pt off 4.5pt}]  (187,74) -- (182,137) ;
%Straight Lines [id:da07441502531747934] 
\draw    (207,98.5) -- (182,137) ;
%Straight Lines [id:da9713533074330818] 
\draw    (404,47) -- (594,95) ;
%Straight Lines [id:da9983958496160543] 
\draw    (594,95) -- (438,94) ;
%Straight Lines [id:da9077478569786941] 
\draw [color={rgb, 255:red, 208; green, 2; blue, 27 }  ,draw opacity=1 ]   (362,183) -- (594,95) ;
%Straight Lines [id:da697127419294798] 
\draw    (404,47) -- (438,94) ;
%Straight Lines [id:da28726018388459496] 
\draw    (438,94) -- (418,162) ;
%Straight Lines [id:da6779351636983375] 
\draw    (404,47) -- (418,162) ;
%Straight Lines [id:da7469504304401373] 
\draw    (461,61) -- (487,94) ;
%Straight Lines [id:da034435284509110486] 
\draw    (487,94) -- (418,162) ;
%Straight Lines [id:da196024122542642] 
\draw  [dash pattern={on 4.5pt off 4.5pt}]  (461,61) -- (418,162) ;
%Straight Lines [id:da12694360059889553] 
\draw    (496,70) -- (516,94.5) ;
%Straight Lines [id:da827940196559507] 
\draw  [dash pattern={on 4.5pt off 4.5pt}]  (496,70) -- (418,162) ;
%Straight Lines [id:da5128956366131279] 
\draw    (516,94.5) -- (418,162) ;

\end{tikzpicture}

    \caption{Ideal prisms obtained in the cases when $\gamma$ is loxodromic (left) or elliptic (right)}
    \label{fig:idealprisms}
\end{figure}\noindent

\begin{figure}[hbt!]
    \centering

\tikzset{every picture/.style={line width=0.75pt}} %set default line width to 0.75pt        

\begin{tikzpicture}[x=0.75pt,y=0.75pt,yscale=-1,xscale=1]
%uncomment if require: \path (0,665); %set diagram left start at 0, and has height of 665

%Straight Lines [id:da49959180585988405] 
\draw [color={rgb, 255:red, 208; green, 2; blue, 27 }  ,draw opacity=1 ]   (340,20) -- (340,320) ;
%Curve Lines [id:da29473055947835003] 
\draw    (350,20) .. controls (350,209) and (383,313) .. (430,320) ;
%Curve Lines [id:da29984689741798776] 
\draw    (360,20) .. controls (360,209) and (423,313) .. (470,320) ;
%Curve Lines [id:da6446515557465957] 
\draw    (370,20) .. controls (370,209) and (463,313) .. (510,320) ;
%Curve Lines [id:da7272764141772359] 
\draw    (380,20) .. controls (380,209) and (503,313) .. (550,320) ;
%Curve Lines [id:da3183405032080273] 
\draw    (390,20) .. controls (390,209) and (543,314) .. (590,321) ;
%Curve Lines [id:da7031404331751689] 
\draw    (400,20) .. controls (400,209) and (584,313) .. (631,320) ;
%Curve Lines [id:da29406931943479697] 
\draw    (330,20) .. controls (330,209) and (296,314) .. (250,320) ;
%Curve Lines [id:da3825843327613312] 
\draw    (321,20) .. controls (321,209) and (256,314) .. (210,320) ;
%Curve Lines [id:da8742973446706166] 
\draw    (311,20) .. controls (311,209) and (216,314) .. (170,320) ;
%Curve Lines [id:da3069311320839969] 
\draw    (300,20) .. controls (300,209) and (176,314) .. (130,320) ;
%Curve Lines [id:da7366871978976486] 
\draw    (290,20) .. controls (290,209) and (136,314) .. (90,320) ;
%Curve Lines [id:da3072389888924765] 
\draw    (281,20) .. controls (281,209) and (96,314) .. (50,320) ;
%Curve Lines [id:da3625388482414704] 
\draw [color={rgb, 255:red, 74; green, 144; blue, 226 }  ,draw opacity=1 ]   (374,201) .. controls (384,234) and (401,290) .. (423,300) ;
%Curve Lines [id:da2713784781472428] 
\draw [color={rgb, 255:red, 74; green, 144; blue, 226 }  ,draw opacity=1 ]   (382,102) .. controls (388,104) and (400,107) .. (407,110) ;
%Curve Lines [id:da14551139270271363] 
\draw [color={rgb, 255:red, 74; green, 144; blue, 226 }  ,draw opacity=1 ]   (269,111) .. controls (275,109) and (290,104) .. (297,102) ;
%Curve Lines [id:da3419093819804955] 
\draw [color={rgb, 255:red, 74; green, 144; blue, 226 }  ,draw opacity=1 ]   (306,201) .. controls (298,228) and (282,288) .. (253,301) ;
%Straight Lines [id:da9338699095578633] 
\draw [color={rgb, 255:red, 208; green, 2; blue, 27 }  ,draw opacity=1 ]   (340,340) -- (340,640) ;
%Curve Lines [id:da24775450271088628] 
\draw    (350,340) .. controls (350,529) and (383,633) .. (430,640) ;
%Curve Lines [id:da6570747458766506] 
\draw    (360,340) .. controls (360,529) and (423,633) .. (470,640) ;
%Curve Lines [id:da5464730396619863] 
\draw    (370,340) .. controls (370,529) and (463,633) .. (510,640) ;
%Curve Lines [id:da1552783269797191] 
\draw    (380,340) .. controls (380,529) and (503,633) .. (550,640) ;
%Curve Lines [id:da5742980535735449] 
\draw    (390,340) .. controls (390,529) and (543,634) .. (590,641) ;
%Curve Lines [id:da13675553713860533] 
\draw    (400,340) .. controls (400,529) and (584,633) .. (631,640) ;
%Curve Lines [id:da4830373282008422] 
\draw    (330,340) .. controls (330,529) and (296,634) .. (250,640) ;
%Curve Lines [id:da9657003505020871] 
\draw    (321,340) .. controls (321,529) and (256,634) .. (210,640) ;
%Curve Lines [id:da8858145989153057] 
\draw    (311,340) .. controls (311,529) and (216,634) .. (170,640) ;
%Curve Lines [id:da0741412672332109] 
\draw    (300,340) .. controls (300,529) and (176,634) .. (130,640) ;
%Curve Lines [id:da8774442750111213] 
\draw    (290,340) .. controls (290,529) and (136,634) .. (90,640) ;
%Curve Lines [id:da9590232408812169] 
\draw    (281,340) .. controls (281,529) and (96,634) .. (50,640) ;
%Curve Lines [id:da17417838876328728] 
\draw [color={rgb, 255:red, 74; green, 144; blue, 226 }  ,draw opacity=1 ]   (374,521) .. controls (384,554) and (401,610) .. (423,620) ;
%Curve Lines [id:da7350024456337265] 
\draw [color={rgb, 255:red, 74; green, 144; blue, 226 }  ,draw opacity=1 ]   (374,521) .. controls (383,524) and (415,536) .. (424,540) ;
%Curve Lines [id:da5255809390339643] 
\draw [color={rgb, 255:red, 74; green, 144; blue, 226 }  ,draw opacity=1 ]   (382,422) .. controls (388,424) and (400,427) .. (407,430) ;
%Curve Lines [id:da38245423395554234] 
\draw [color={rgb, 255:red, 74; green, 144; blue, 226 }  ,draw opacity=1 ]   (269,431) .. controls (275,429) and (290,424) .. (297,422) ;
%Curve Lines [id:da15536710411871058] 
\draw [color={rgb, 255:red, 74; green, 144; blue, 226 }  ,draw opacity=1 ]   (256,540) .. controls (269,534) and (298,524) .. (306,521) ;
%Curve Lines [id:da6871594737917559] 
\draw [color={rgb, 255:red, 74; green, 144; blue, 226 }  ,draw opacity=1 ]   (306,521) .. controls (298,548) and (280,608) .. (253,621) ;
%Shape: Polygon Curved [id:ds6643561361824244] 
\draw  [color={rgb, 255:red, 74; green, 144; blue, 226 }  ,draw opacity=1 ][fill={rgb, 255:red, 206; green, 232; blue, 250 }  ,fill opacity=0.55 ] (298,102) .. controls (321,102) and (358,101) .. (383,102) .. controls (387,130) and (407,196) .. (425,220) .. controls (409,214) and (387,208) .. (374,201) .. controls (365,199) and (323,202) .. (306,201) .. controls (293,205) and (270,215) .. (257,220) .. controls (274,194) and (297,124) .. (298,102) -- cycle ;
%Shape: Polygon Curved [id:ds5421372184305704] 
\draw  [color={rgb, 255:red, 74; green, 144; blue, 226 }  ,draw opacity=1 ][fill={rgb, 255:red, 206; green, 232; blue, 250 }  ,fill opacity=0.44 ] (360,340) .. controls (365,340) and (375,340) .. (380,340) .. controls (380,372) and (384,404) .. (388,424) .. controls (386,424) and (388,424) .. (382,422) .. controls (383,436) and (410,527) .. (424,540) .. controls (411,534) and (396,530) .. (385,525) .. controls (371,486) and (360,408) .. (360,340) -- cycle ;
%Shape: Polygon Curved [id:ds6564715860369872] 
\draw  [color={rgb, 255:red, 74; green, 144; blue, 226 }  ,draw opacity=1 ][fill={rgb, 255:red, 206; green, 232; blue, 250 }  ,fill opacity=0.44 ] (321,340) .. controls (326,340) and (295,340) .. (300,340) .. controls (300,372) and (296,403) .. (291,424) .. controls (296,423) and (295,422) .. (297,422) .. controls (295,452) and (271,515) .. (256,540) .. controls (268,535) and (280,531) .. (295,525) .. controls (310,482) and (321,408) .. (321,340) -- cycle ;

\end{tikzpicture}
    
    \caption{}
    \label{fig:idealchange}
\end{figure}\noindent

The rectangles $R^t_1,\ldots,R^t_m$ are taken in such a way that each closed curve $\gamma$ in $\Gamma$ is covered by the closure of a single rectangle $R^t_i$, which we label by $R^t_\gamma$ (see Figure \ref{fig:rectangles1}). In the case that $\gamma$ is elliptic the vertical sides are not well-defined segments, but we have a well-defined pair of cusped cylinders which union we still denoted by $R^0_\gamma$.  Observe that it is only in these rectangles where $\lambda_0\cap R^t_i$ is not a finite union of geodesic segments (see Figure \ref{fig:rectangles2} for how to cover these finitely many segments by rectangles). The horizontal collapse of such $R^t_\gamma$ is a closed edge homotopic to the curve $\gamma$ covered by $R^t_\gamma$ (see Figure \ref{fig:graph} for how to complete to a triangulation on each pair of pants), from where we will choose that the $\rho_t$-equivariant polyhedral maps $g_t:\tilde{S} \rightarrow \mathbb{H}^3$ send these closed edges to the axis of $\rho_t(\gamma)$. The potential issues of choosing $g_t$ in such a way are that (1) the polyhedral decomposition of the homotopy $H_t$ might include degenerate polyhedra and (2) how to make sense of the length of $b'$ in the Bonahon-Schl\"afli formula.

For (1) the only $3$-chains that could degenerate are the ones concerning $\gamma$. In preparation, take $R^0_\gamma$ such that each horizontal side extends to a geodesic ray with the same endpoint as $\gamma$. In the axis $\tilde{\gamma}$ take a point $A$ so that any $\gamma$ translate of $A$ does not belong to either $\tilde{\lambda_0}$ or to any lifting of the horizontal sides of $R^0_\gamma$. This is possible because we are in the open set where $\tilde{\gamma}$ is not included in $\tilde{\lambda_0}$. Use $A$ as the vertex in $\gamma$ of the triangulation of Figure \ref{fig:graph}. Since the horizontal sides of $R^0_\gamma$, $\tilde{\gamma}$ and $\lambda_0\cap R^0_\gamma$ have all a point at infinity in common, the associated $3$-chains to $R^0_\gamma$ are in fact prisms. This is the combinatorial type we will take for these $3$-chains. Such prisms are not necessarily non-degenerate, as if $\gamma$ is elliptic then a side of the prism will collapse to a point. Regardless, since by choice the translates of $A$ are disjoint from the boundary of the opposite rectangle, the planes containing a face of the prism are all well-defined. Similarly, each edge of the prism is contained in a well-defined geodesic line (in Figure \ref{fig:prisms} the collapsed segment belongs to $\tilde{\gamma}$). Hence, even if the prism degenerates, we have a well-defined notion of angles between adjacent faces, where we allowed the values of $0,\pi$ in $\mathbb{R}/2\pi\mathbb{Z}$ for angles and $0$ for lengths. Same follows for the $3$-chains bounded by $g_0\circ\tilde{h}_0$, and we extend this configuration for $t$ small.

For (2) Bonahon observes that $\frac12\ell(b'_t)$ in $R^t_\gamma$ is equal to the variations of volume of the $3$-chains obtained from $R^t_\gamma$ minus the edge contributions of edges that do not belong to the lamination. More specifically, Bonahon proves that the contribution to the Schl\"afli formula of the edges in the lamination add up to $\frac12\ell(b'_t)$, while the sum of contributions of all other edges cancel out while adding all $R^t_i$ together. So for our particular construction we will rearrange the rectangles of $R^t_\gamma\setminus \lambda^t_0$ for the top to the bottom diagram of Figure \ref{fig:idealchange}. Doing so will stack all prisms to form finitely many tetrahedra with an ideal vertex such as in Figure \ref{fig:idealprisms}. Polyhedra with ideal vertices have as well a Schl\"afli formula, where one take an horoball at each ideal vertex and computes the (signed) length of an edge by taking the finite segment in the edge that goes between vertices/horoballs. Then the sum of (edge length)(derivative of dihedral angle) is well defined and independent of the family of horoballs at ideal vertices.

Now we have to subtract the contribution of edges that are not in the lamination $\tilde{\lambda_t}$. Edges joining $\tilde{\gamma}$ to the rest of the prism already cancel out once we assembled the ideal tetrahedra, with the exception of the finite face of the tetrahedra, which we exclude. Hence we are left with the contributions of the edges at the ideal vertex. Observe that from the change of orientation of the tetrahedra from each cuff side of $\gamma$, the dihedral angles at $\tilde{\gamma}$ cancel out, as along as we consider the same horoball for all ideal tetrahedra. Hence we are left with the edges coming from the bending.

On the other hand, $\ell(b'_t)$ over a region is obtained by integrating the horizontal length agains the measure induced by the change of bending cocycle. When $\gamma$ is elliptic the pleated surface is of finite type, which makes all the lengths infinite (they correspond to geodesic rays towards the fixed endpoint of $\tilde{\gamma}$). But the horoball of the previous paragraph corresponds to an horocycle in the pleated surface, where we now we can compute length of a segment from the horocycle, and multiply that by the corresponding bending. This sum will be well-defined by the same reasons as the Schl\"afli formula for ideal polyhedra, as long as we take matching horocycles at each side of $\gamma$, namely horocycles coming from the same horoball at the fixed endpoint of $\hat{\gamma}$. 

Putting all together, both $V'_t$ and $\frac12\ell(b'_t)$ have the same contribution on the rectangle $R^t_\gamma$, namely both contributions are equal to 
\[\sum_{e \text{ is an ideal edge in } \lambda_0} \frac12\ell(e).\theta'(e)\]
where $\ell$ denotes length of the segment of the edge $e$ in $R^t_\gamma$ outside a fixed horoball based at the endpoint of $\tilde{\gamma}$, and $\theta(e)$ is the exterior dihedral angle. With these conventions we have justified
\[V'_t = \frac12\ell(b'_t)
\]
for all cases, which is sufficient for the purpose of this article, since what we need is to characterize the derivative of volume as depending only on the boundary information.

\section{Proof of Main theorem}\label{sec:mainproof}

Now that we have all terms defined, let us restate our main result.

\begin{theorem}\label{thm:main}
Let $M$ be a hyperbolizable compact 3-manifold with boundary. Let $\chi_0(M)$ be the connected irreducible component of the discrete and faithful representations. Then the map $i_*:\chi_0(M)\rightarrow \chi(\partial M)$ is a birational isomorphism onto its image.
\end{theorem}

Since we want to use the Bonahon-Schl\"afli formula and volume rigidity, we first need a lemma saying that generically our pleated surface construction is well-defined.

\begin{lem}\label{lemma:pleatedreps}
Let $M$ be a hyperbolizable compact 3-manifold with boundary, and let $\lambda$ be a maximal geodesic lamination on $\partial M$ that contains a pants decomposition. Let $\chi_0(M)$ be the connected component of the discrete and faithful representations. Then the set 
\[\mathcal{P}(M, \lambda)=\lbrace [\rho]\in\chi_0(M)\,|\, (\partial M,\rho) \text{ has a pleated surface with pleating locus } \lambda  \rbrace\]

is an open Zariski dense set in $\chi_0(M)$.
\end{lem}

\begin{proof} (Fixing a particular $\lambda$)

We follow the description of Thurston \cite{Thurston}. Let $S\subseteq \partial M$ be a connected component with some fixed hyperbolic structure. Fix then a pants decomposition $\mathcal{P}=\lbrace \gamma_i\rbrace$ and extend it by finitely many geodesic lines to a maximal geodesic lamination $\lambda$. Then a given $ [\rho]\in\chi_0(M)$ has a pleating surface with pleating locus $\lambda$ if:

\begin{enumerate}
    \item\label{1stcondition} $\rho(\gamma_i)$ is non-trivial and non-parabolic for all $\gamma_i\in\mathcal{P}$.
    \item\label{2ndcondition}The end points of $\rho(\gamma_i),\,\rho(\gamma_j)$ are distinct for any $\gamma_i\neq\gamma_j\in \mathcal{P}$.
\end{enumerate}

Indeed, if these conditions are satisfied, the lifts of $\mathcal{P}$ in $\tilde{\lambda}$ can be mapped equivariantly to $\mathbb{H}^3$ by choosing the geodesic representatives $\rho(\gamma_i)$. And since any line $\ell \in\lambda\setminus\mathcal{P}$ accumulates to $\gamma_i\neq\gamma_j\in \mathcal{P}$ with different endpoints, then we can map a lift of $\ell$ to the unique geodesic joining distinct endpoints of the lifts of $\gamma_i, \gamma_j$. Hence for any ideal triangle in $S\setminus\lambda$ we have a map of its boundary to $M$, so there exists a corresponding ideal triangle in $M$. Such ideal triangles will make the realization of $\lambda$ in $M$.

What is left to see is that (\ref{1stcondition}) and (\ref{2ndcondition}) contain a Zariski dense set. The negative of $(\ref{1stcondition})$ corresponds to $tr^2(\rho(\gamma_i))=4$, which is an polynomial equation on the coefficients of $\rho(\gamma_i)$. Similarly, if $\rho(\gamma_i),\rho(\gamma_j)$ share an endpoint then $tr^2(\rho(\gamma_i)\rho(\gamma_j)\rho(\gamma_i)^{-1}\rho(\gamma_j)^{-1})=4$, which is an algebraic equation on the coefficients of the commutator of $\rho(\gamma_i),\rho(\gamma_j)$. Finally, since these equations are not satisfied for convex co-compact representations, we have that the negatives of (\ref{1stcondition}), (\ref{2ndcondition}) are finite unions of proper algebraic sets, hence the intersection of their complements is a connected Zariski dense set in $\chi_0(M)$.

\end{proof}

Given that we have that the peripheral map $i_*:\chi_0(M) \rightarrow \overline{i_*(\chi_0(M))}\subset \chi_0(\partial M)$ is dominant, we can use the following lemma to show that $i_*$ is essentially a finite-to-1 covering. Then the main theorem will follow from showing that $i_*$ is essentially injective.

\begin{lem}[\cite{Harris92} Proposition 7.16 ]\label{lemma:k-to-1}
Let $X,Y$ be (complex) affine algebraic varieties, and let $f:X\rightarrow Y$ be a dominant rational map. If $X,Y$ have the same dimension, then there exist open Zariski dense subsets $X_0\subseteq X,\,Y_0\subseteq Y,\, f^{-1}(Y_0)=X_0$ and integer $k$ so that the map $f|_{X_0}:X_0\rightarrow Y_0$ is $k$-to-$1$.
\end{lem}

The idea to show that $k=1$ for $i_*$ is to use volume rigidity for characters. In order to do so, we need to show first that volume only depends on the peripheral data.

\begin{lem}\label{lemma:samevolume}
Let $M$ be a hyperbolizable compact 3-manifold with boundary, and let $\Gamma$ be an unoriented pants decomposition on $\partial M$. Let $\chi_0(M)$ be the connected component of the discrete and faithful representations. There exists an open Zariski dense subset $Z\subseteq \chi_0(M)$ so the following holds
\begin{enumerate}[label=(\alph*)]
    \item\label{volumedefine} $vol_\Gamma:Z\rightarrow \mathbb{R}$ given by
    \[vol_\Gamma(\rho) = \sum_{\Gamma' \text{ orientation on } \Gamma} vol_{\lambda_0(\Gamma')}(\rho)
    \]
    is well-defined.
    \item\label{boundaryvolume} On $W=i_*(Z)$ there is a well-defined map $V:W\rightarrow\mathbb{R}$ so the diagram commutes
        \[
            \begin{tikzcd}
            Z \arrow{r}{i_*} \arrow[swap]{dr}{vol_\Gamma} & W \arrow{d}{V} \\
            & \mathbb{R}
        \end{tikzcd}
        \]
\end{enumerate}
\end{lem}

\begin{proof}
On the set $\mathcal{P}(M, \lambda_0(\Gamma))$ we can define $vol_{\lambda_0(\Gamma)}$ as the volume interior to the pleated surface with pleating locus $\lambda_0(\Gamma)$. By the arguments explained in Subsection \ref{subsec:volumevariation} this is a well-defined continuous function, although potentially non-differentiable. The reason for this is because we require a smooth equivariant family of endpoints for the lifts of $\Gamma$ in order to apply the Bonahon-Schl\"afli formula. As explained in Subsection \ref{subsec:pleated}, this is equivalent to choose an orientation $\Gamma'$ of $\Gamma$. Because we have to consider the case when $\rho(\gamma)$ is elliptic for some $\gamma\in \Gamma$, we cannot in principle choose a global smooth equivariant family of endpoints for the lifts of $\Gamma$. Instead, we choose all possible orientations at once and take the sum to obtain the smooth function $vol_\Gamma$, whose derivative is given by the sum of Bonahon-Schl\"afli formulas for each orientation. 

By intersecting with the Zariski dense sets of Lemma \ref{lemma:k-to-1} we can assume that we have $Z\subseteq \chi_0(M),\, W=i_*(Z)\subseteq \chi_0(\partial M)$ so that \ref{volumedefine} is satisfied and $i_*:Z\rightarrow W$ is $k$-to-$1$. Our goal is to show that for any $\rho_1, \rho_2 \in Z$ with $i_*(\rho_1)=i_*(\rho_2)$ we have that $vol_\Gamma(\rho_1)=vol_\Gamma(\rho_2)$. Since $Z$ is connected, we can take a path $\rho_t$ in $Z$ with endpoints $\rho_1,\rho_2$. Observe that since the Bonahon-Schl\"afli formula depends exclusively on peripheral information, the change of volume $vol_\Gamma$ on any lift of $i_*(\rho_t)$ through $i_*:Z\rightarrow W$ is always equal to $vol_\Gamma(\rho_2)-vol_\Gamma(\rho_1)$. Concatenate then consecutive lifts. Since the we have a finite fiber, these consecutive lifts must contain a close loop. Then the change of volume on that close loop is equal to $0$, but it is also a multiple of $vol_\Gamma(\rho_2)-vol_\Gamma(\rho_1)$. Then we must have that $vol_\Gamma(\rho_2)=vol_\Gamma(\rho_1)$, from where (\ref{boundaryvolume}) follows.

%Now we define $V:W\rightarrow\mathbb{R}$ by:

%\[V(\eta) = \frac1k \sum_{i_*(\rho)=\eta}vol_\Gamma(\eta)
%\]

%By work of Bonahon \cite{Bonahon98}[Theorem 3], given a differentiable 1-parameter family in $\mathcal{P}(M, \lambda)$, the derivative of $vol_\lambda$ is given by
%\[vol_\lambda' = \frac12 \ell(b'),
%\]
%where $b'\in\mathcal{H}(\lambda,\mathbb{R})$ is the derivative of the bending cocycle $b\in\mathcal{H}(\lambda,\mathbb{R}/2\pi\mathbb{Z})$ of the pleated surface with pleated locus $\lambda$. Since the bending cocycle is determined by the representation induced by the pleated surface, it follows that the Bonahon-Schl\"afli formula descends to $V$ as

%\[V' = \frac12 \ell(b')
%\]

%Fix $\eta\in W$. Since $Z$ is connected, we can join $\rho_1,\rho_2\in i_*^{-1}(\eta)$ by a differentiable path $\alpha$. Since $i_*(\alpha)$ is a closed path we have that the change of $V$ between its endpoints is $0$. Since the Bonahon-Schl\"afli formula agrees for $\alpha, i_*(\alpha)$, this implies that $vol_\lambda(\rho_1) = vol_\lambda(\rho_2)$. Since $\eta\in W$ was arbitrary, \ref{boundaryvolume} follows.

\end{proof}

Now we are ready to prove the Main theorem through volume rigidity.

\begin{proof}[Proof of Main Theorem]

By Lemmas \ref{lemma:pleatedreps}, \ref{lemma:k-to-1}, \ref{lemma:samevolume} we have Zariski open subsets $Z\subseteq \chi_0(M)$, $W\subseteq\chi_0(\partial M)$ so that $i_*:Z\rightarrow W$ is a $k$-to-$1$ map and $vol_\Gamma$ is constant over the fibers of $i_*$. By a result of Brooks \cite[Theorem 1]{Brooks86}, there is a dense set $E\subset Z$ of convex co-compact characters that admit a co-compact extension by reflections. This extension, known also as the Thurston orbifold trick, is a co-finite extension made by considering system of orthogonal planes on each geometrically finite end and extend by their reflections.

Take then $\chi_\rho\in E$ and $\chi_{\rho'}\in Z$ so that $i_*(\chi_\rho)=i_*(\chi_{\rho'})$. Then there exists $G > \pi_1(M)$ with $[G:\pi_1(M)]$ finite and $\tilde{\rho}\in R(G)$ so that $\tilde{\rho}|_{\pi_1(M)} =\rho$ and $\mathbb{H}^3/\tilde{\rho}(G)$ is a compact hyperbolic $3$-manifold. Since $i_*(\rho') = i_*(\rho)$, we can find extension $\tilde{\rho'}\in R(G)$ of $\rho'$. This is because $\rho'$ coincides (up to conjugation) as a representation with $\rho$ in each end, and the extension was made by reflecting on the system of orthogonal planes. And since $vol_\Gamma(\rho)= vol_\Gamma(\rho')$, then the volume of the complements of the reflecting planes in $\rho, \rho'$ are the same. This is because the defects between any two summmands of $vol_\Gamma$ or between the system of orthogonal planes and a summand of $vol_\Gamma$ are determined by the representation of $\pi_1(\partial M)$, where $\rho$ and $\rho'$ coincide. But then this implies that the representations $\tilde{\rho}, \tilde{\rho'}$ have the same volume. As $\tilde{\rho}$ corresponds to a compact hyperbolic $3$-manifold, volume rigidity for compact hyperbolic groups (Gromov-Thurston-Goldman volume rigidity, see \cite{Dunfield99}[Theorem 6.1]) implies that $\tilde{\rho}, \tilde{\rho'}$ are conjugated. Then it follows that $\chi_\rho=\chi_{\rho'}$.

Hence the map $i_*:Z\rightarrow W$ is $1$-to-$1$ in $E\subset Z$, so it has to be that $k=1$. This implies that $i_*$ is $1$-to-$1$, and since we knew that $i_*$ was dominant, then $i_*$ is a birrational isomorphism (see the remark after the definition of birational map on p. 77 and Exercise 7.8 of \cite{Harris92}; this fact uses that we are working over charactertic $0$).

\end{proof}

\textbf{Remarks}
\begin{enumerate}\label{remarks}
    \item In the case when $M$ is small we can say more about the map $i_*:\chi_0(M)\rightarrow \chi(\partial M)$. A $3$-manifold $M$ is \textit{small} if there does not exist incompressible, non-boundary parallel surface $\Sigma\subset M$. In this case we have  that the map $i_*:\chi_0(M)\rightarrow \chi(\partial M)$ is a surjective, finite-to-one map. For if there is a character $\rho\in\chi(\partial M)$ with non-zero dimensional preimage or in the accumulation of the image of $i_*$, then there is an ideal point $p$ on $\chi_0(M)$ so that $i_*(p)=\rho$. By Culler-Shalen theory (done by Boyer-Zhang \cite{BoyerZhang98} for the $\PSLC$ case) meromorphic valuation at the ideal point $p$ produces a $\pi_1(M)$ action on a tree, from where an incompressible surface $\Sigma$ is produced. Since the vertices of the tree are taken by classes of valuation lattices on the field of meromorphic functions times itself, the fact $i_*(p)=\rho$ is well-defined implies that the boundary of $M$ acts trivially on such tree, so $\Sigma$ can't be boundary parallel.
    
    \item We can combine our approach and the work of \cite{Dunfield99}, \cite{KlaffTillmann} to obtain a similar statement for $M$ geometrically finite. Let $(M,\mathcal{C})$ be a geometrically finite hyperbolic $3$-manifold, where $\mathcal{C}$ is the collection of conjugacy classes in $\partial M$ corresponding to the rank-$1$ cusps. Denote by $\partial_\mathcal{C}M$ the boundary of $M$ after pinching the generators of $\mathcal{C}$. Then we can define the representation and character varieties $ R(M,\mathcal{C}),\chi(M,\mathcal{C})$ as the subvarieties of $R(M), \chi(M)$ restricted to the condition that $\mathcal{C}$ are always mapped to parabolic elements in $\PSLC$. Taking $\chi_0(M,\mathcal{C})$ as the irreducible component containing geometrically finite characters pinched at $\mathcal{C}$, then the map
    \[i_*:\chi_0(M,\mathcal{C}) \rightarrow \chi(\partial_\mathcal{C}M)
    \]
    is a birational isomorphism with is image. Rank-$2$ cusps are dealt as in \cite{Dunfield99}, \cite{KlaffTillmann}, while for our pleated surface construction we extend generators of $\mathcal{C}$ to a pants decomposition of $\partial M$. Then the choice of endpoint of ideal triangles at a lift of an element of $\mathcal{C}$ is given by the unique parabolic fixed point.
\end{enumerate}

\bibliographystyle{amsalpha}
\bibliography{mybib}

\end{document}